\documentclass[a4paper,leqno]{amsart}

\usepackage[english]{babel} 									%Support for many languages
\usepackage[T1]{fontenc}							%The last two are related to making "ä" etc. more usable in Latex
\usepackage[utf8]{inputenx}
				
\usepackage{mathtools}								%Includes amsmath, basic math commands
\usepackage{empheq}									%Useful for systems of equations
\usepackage{enumitem}								%Improves enumerate environment among other things

\usepackage{amssymb}								%Fonts and mathematical symbols
\usepackage[sc]{mathpazo}							%Fonts		
\usepackage{mathrsfs}								%\mathscr command; http://www.math.washington.edu/~lee/Courses/typesetting-script.pdf

\usepackage{verbatim}                               %comment environment: hide everything
\usepackage{iflang}									%Used in theoremenv.sty
\usepackage{xifthen}								%See http://tex.stackexchange.com/questions/58628/optional-argument-for-newcommand
\usepackage[final]{pdfpages}						%Enables including multiple pages of a pdf easily.
\usepackage{todonotes}								%\todo{note} \usepackage[disable]{todonotes}
\usepackage{aliascnt}								%Useful for making a counter share the number with other counter, but with different names.. used in thereomenv.sty in amsthm 																								 variants of the code
\usepackage{needspace}								%Needed in theoremenv, checks that proofs etc. don't get separated too easily: http://tex.stackexchange.com/a/30125
\usepackage{refcount}

\usepackage{subcaption}			%Naming figures (subfigure)
\usepackage{placeins} 			%Restricting movement of floats (FloatBarrier)
\usepackage{grffile}			%More possibilities with naming figures
\usepackage[all]{xy}			%Useful in Algebraic Topology

\usepackage{transparent}
\usepackage{color}				%Colorful pictures

\usepackage{graphicx}			%More options for includegraphic. See http://ctan.org/pkg/graphicx for example.

\usepackage[babel]{csquotes}							%More options for \cite etc.					

\usepackage{nameref}									%I.e. can refer to Theorem (Euler's theorem) as "Euler's theorem" with a hyperlink
\usepackage[colorlinks	=	true,						%Hyperlinks to references (an absolute must have)
          	 	linkcolor	=	red,
            	urlcolor	=	red,
            	citecolor	=	red]%
            	{hyperref}
\usepackage{bookmark}									%Hyperref options; used in theoremenv to bookmark theorems in pdf files automatically
\usepackage{cleveref}									%Further options for hyperref, \cref and \Cref are interesting (and their multiple variants)
\theoremstyle{definition}
\newtheorem{thm}{Theorem}[section]

\newtheorem{defi}[thm]{Definition}
\newtheorem{lemm}[thm]{Lemma}

\newtheorem{rem}[thm]{Remark}

\newtheorem{prop}[thm]{Proposition}

\theoremstyle{remark}

\DeclareMathOperator{\Mod}{mod}

\DeclareMathOperator{\diam}{diam}

\DeclareMathOperator{\loc}{loc}

\allowdisplaybreaks

\newcommand{\norm}[1]{ \left\Vert #1 \right\Vert }	%Norm
\newcommand{\abs}[1]{ \left| #1 \right| }           %Absolute values
\newcommand{\apmd}[2][]{							%Approximate Metric Derivative with options
	\ifthenelse{\equal{#1}{}}%
					{ \operatorname{N}_{#2}	}%
					{ \operatorname{N}_{#1,#2} 	}}

\begin{document}
\title[Quasiconformal Jordan domains]{Quasiconformal Jordan domains}

\author{Toni Ikonen}

\address{University of Jyvaskyla \\ Department of Mathematics and Statistics \\
P.O. Box 35 (MaD) \\
FI-40014 University of Jyvaskyla}
\email{toni.m.h.ikonen@jyu.fi}


%\date{\today}

\begin{abstract}
We extend the classical Carathéodory extension theorem to quasiconformal Jordan domains $( Y, d_{Y} )$. We say that a metric space $( Y, d_{Y} )$ is a \emph{quasiconformal Jordan domain} if the completion $\overline{Y}$ of $( Y, d_{Y} )$ has finite Hausdorff $2$-measure, the \emph{boundary} $\partial Y = \overline{Y} \setminus Y$ is homeomorphic to $\mathbb{S}^{1}$, and there exists a homeomorphism $\phi \colon \mathbb{D} \rightarrow ( Y, d_{Y} )$ that is quasiconformal in the geometric sense.

We show that $\phi$ has a continuous, monotone, and surjective extension $\Phi \colon \overline{ \mathbb{D} } \rightarrow \overline{ Y }$. This result is best possible in this generality. In addition, we find a necessary and sufficient condition for $\Phi$ to be a quasiconformal homeomorphism. We provide sufficient conditions for the restriction of $\Phi$ to $\mathbb{S}^{1}$ being a quasisymmetry and to $\partial Y$ being bi-Lipschitz equivalent to a quasicircle in the plane.
\end{abstract}

\maketitle\thispagestyle{empty}

\section{Introduction}\label{sec:intro}
Let $(X, d_{X} )$ be a metric space with locally finite Hausdorff $2$-measure. If $X$ is also homeomorphic to a $2$-manifold, we say that $( X, d_{X} )$ is a \emph{metric surface}. A homeomorphism $\phi \colon ( X, d_{X} ) \rightarrow ( Y, d_{Y} )$ between metric surfaces is \emph{quasiconformal} if there exists $K \geq 1$ such that for all path families $\Gamma$,
\begin{equation}
    \label{eq:QC}
    K^{-1}\Mod \Gamma
    \leq
    \Mod \phi \Gamma
    \leq
    K \Mod \Gamma,
\end{equation}
where $\Mod \Gamma$ is the \emph{conformal modulus} of $\Gamma$, see \Cref{sec:sobolev}.

We say that a metric surface $( Y, d_{Y} )$ is a \emph{metric Jordan domain} if the metric completion $\overline{Y}$ is homeomorphic to the closed unit disk $\overline{ \mathbb{D} }$, the \emph{boundary} $\partial Y = \overline{ Y } \setminus Y$ is homeomorphic to the unit circle $\mathbb{S}^{1}$, and the Hausdorff $2$-measure of $\overline{Y}$ is finite.

A metric Jordan domain is a \emph{quasiconformal Jordan domain} if there exists a quasiconformal homeomorphism $\phi \colon \mathbb{D} \rightarrow ( Y, d_{Y} )$. A metric Jordan domain is a quasiconformal one if and only if $( Y, d_{Y} )$ is \emph{reciprocal} as introduced in \cite[Theorem 1.4]{Raj:17}; see \Cref{defi:reciprocal}. This uses the facts that $\mathcal{H}^{2}_{ \overline{Y} }( \overline{Y} ) < \infty$ and that $\partial Y$ is a non-trivial continuum.

In general, it is not true that the completion $\overline{ Y }$ of a quasiconformal Jordan domain is a quasiconformal image of the closed unit disk $\overline{ \mathbb{D} }$. We illustrate this with an example after \Cref{thm:carat:metric}. Contrast this with the classical case when $Y$ is a Jordan domain in the plane $\mathbb{R}^{2}$. Then any $1$-quasiconformal homeomorphism $\phi \colon \mathbb{D} \rightarrow Y$, i.e., any Riemann map from the unit disk onto $Y$ extends to a homeomorphism $\Phi \colon \overline{\mathbb{D}} \rightarrow \overline{Y}$ by a result known as the Carathéodory extension theorem \cite[Chapter I, Theorem 3.1]{Gar:Mar:05}. In fact, the extension still satisfies \eqref{eq:QC} with $K = 1$. Additionally, if $\phi \colon \mathbb{D} \rightarrow Y$ is $K$-quasiconformal for some $K \geq 1$, it still has a homeomorphic extension to the boundary.

\subsection{Carathéodory's theorem}\label{sec:cara}
We prove the following generalization of the classical Carathéodory extension theorem of quasiconformal maps.
\begin{thm}\label{thm:carat:metric}
Let $\phi \colon \mathbb{D} \rightarrow Y$ be a quasiconformal map onto a quasiconformal Jordan domain. Then there exists an extension $\Phi \colon \overline{ \mathbb{D} } \rightarrow \overline{Y}$ of $\phi$ that is surjective, monotone and $\Phi( \mathbb{S}^{1} ) = \partial Y$.
\end{thm}
Here we say that a map is \emph{monotone} if it is continuous and the preimage of every point is a \emph{continuum}, i.e., a compact and connected set.

The map $\Phi$ might fail to be a homeomorphism. As an example, consider the length space $X$ homeomorphic to $\mathbb{R}^{2}$ obtained by collapsing the Euclidean square $\left[ 0, 1 \right]^{2}$ in $\mathbb{R}^{2}$ to a point. Let $\pi \colon \mathbb{R}^{2} \rightarrow X$ denote the associated $1$-Lipschitz quotient map. We define $Y = \pi( \left( 1, 2 \right) \times \left(0, 1\right) )$. Then $\partial Y = \pi( \partial \left[ 1, 2 \right] \times \left[0, 1 \right] )$. The restriction of $\pi$ to $\left( 1, 2 \right) \times \left(0, 1\right)$ is a 1-quasiconformal map, but its extension collapses the arc segment $\left\{ 1 \right\} \times \left[0, 1 \right]$ to the singleton $\pi( \left[0, 1 \right]^{2} )$. By considering a Riemann map $f \colon \mathbb{D} \rightarrow (0,1)^{2}$, the claim follows by setting $\phi = \pi \circ f$.

Next, we investigate when the extension in \Cref{thm:carat:metric} is a quasiconformal homeomorphism. To this end, for every $y \in \overline{Y}$ and $\diam \overline{Y} \geq R > r > 0$, we let $\Gamma( \overline{B}_{ \overline{Y} }( y, r ), \overline{Y} \setminus B_{ \overline{Y} }( y, R ); \overline{Y} )$ denote the family of paths joining $\overline{B}_{ \overline{Y} }( y, r )$ to $\overline{Y} \setminus B_{ \overline{Y} }( y, R )$.
\begin{prop}\label{prop:QS}
The extension $\Phi$ in \Cref{thm:carat:metric} is quasiconformal if and only if for every $y \in \partial Y$ and $R > 0$ for which $\overline{Y} \setminus B_{ \overline{Y} }( y, R ) \neq \emptyset$,
\begin{equation}
    \label{eq:pointshavezeromod}
    \lim_{ r \rightarrow 0^{+} }
    \Mod
    \Gamma( 
        \overline{B}_{ \overline{Y} }( y, r ),
        \overline{Y} \setminus B_{ \overline{Y} }( y, R );
        \overline{Y}
    )
    =
    0.
\end{equation}
Moreover, if \eqref{eq:pointshavezeromod} holds at each $y \in \partial Y$ and $\phi$ is $K$-quasiconformal, then $\Phi$ is $K$-quasiconformal.
\end{prop}
A well-known fact is that if there exists $C_{U} > 0$ such that for all $y \in \partial Y$ and $0 < r < \diam \partial Y$,
\begin{equation}
    \label{eq:area:growth:intro}
    \mathcal{H}^{2}_{ \overline{Y} }( \overline{B}_{ \overline{Y} }( y, r ) )
    \leq
    C_{U} r^{2},
\end{equation}
then \eqref{eq:pointshavezeromod} holds; see \Cref{lemm:points}. The condition \eqref{eq:pointshavezeromod} has a close link to the reciprocality condition introduced in \cite{Raj:17}; see \Cref{defi:reciprocal}. The aforementioned example of the collapsed disk $\left[0, 1\right]^{2}$ fails \eqref{eq:pointshavezeromod} at exactly one point.

It can happen that the extension $\Phi$ in \Cref{thm:carat:metric} is a homeomorphism, but not quasiconformal; see \cite[Example 6.1]{Iko:Rom:20:acc}. There we have a metric space $X$ for which there exists a $1$-Lipschitz homeomorphism $\pi \colon \mathbb{R}^{2} \rightarrow X$ which is $1$-quasiconformal outside a Cantor set $K \subset \left[0,1\right] \times \left\{0\right\}$, but $\pi|_{ (0,1)^{2} }$ does not extend to a $1$-quasiconformal homeomorphism on $\left[0, 1\right]^{2}$. The claim follows by setting $Y = \pi( (0,1)^2 )$ and setting $\phi = \pi \circ f$ for any Riemann map $f \colon \mathbb{D} \rightarrow (0,1)^{2}$.

\subsection{Quasicircles}
Consider a quasiconformal Jordan domain $Y$ whose boundary points satisfy the area growth inequality \eqref{eq:area:growth:intro}. We know from \Cref{prop:QS} that the extension $\Phi \colon \overline{ \mathbb{D} } \rightarrow \overline{ Y }$ of any quasiconformal homeomorphism $\phi \colon \mathbb{D} \rightarrow Y$ is a quasiconformal homeomorphism. In particular, the \emph{boundary map} $g_{\phi} = \Phi|_{ \mathbb{S}^{1} } \colon \mathbb{S}^{1} \rightarrow \partial Y$ is a homeomorphism.

We are especially interested when we can deduce that $\partial Y$ is a \emph{quasicircle}, i.e., a quasisymmetric image of $\mathbb{S}^{1}$. We refer the reader to \Cref{sec:preliminaries} for definitions.
\begin{thm}[Beurling--Ahlfors extension]\label{thm:QS}
Suppose that $Y$ is a quasiconformal Jordan domain whose boundary points satisfy the area growth \eqref{eq:area:growth:intro}.

If $\phi \colon \mathbb{D} \rightarrow Y$ is a quasiconformal homeomorphism, then the boundary map $g_{ \phi }$ is a quasisymmetry if and only if $\partial Y$ has bounded turning. If $\partial Y$ has bounded turning, then any quasisymmetry $g \colon \mathbb{S}^{1} \rightarrow \partial Y$ is the boundary map of some quasiconformal map $\phi \colon \mathbb{D} \rightarrow Y$.
\end{thm}

\Cref{thm:QS} has a parallel in the classical literature. Ahlfors and Beurling proved in \cite{Ah:Beu:56} that every quasisymmetry $g \colon \mathbb{S}^{1} \rightarrow \mathbb{S}^{1}$ is the boundary homeomorphism of some quasiconformal map $\phi \colon \mathbb{D} \rightarrow \mathbb{D}$. In fact, we apply their result in proving that quasisymmetries $g \colon \mathbb{S}^{1} \rightarrow \partial Y$ extend like claimed. It is also known that the boundary homeomorphisms of quasiconformal maps $\phi \colon \mathbb{D} \rightarrow \mathbb{D}$ are quasisymmetries. So we also recover this result with the assumptions of \Cref{thm:QS}.

We now know from \Cref{thm:QS} that $\partial Y$ is a quasicircle in some situations. We are interested whether or not $\partial Y$ can be bi-Lipschitz embedded into the plane. We say that a quasicircle $Z$ is \emph{planar}, or a \emph{planar quasicircle}, if there exists a bi-Lipschitz embedding $h \colon Z \rightarrow \mathbb{R}^{2}$.

One of the main results obtained in \cite{Her:Mey:12} states that a quasicircle is planar if and only if its \emph{Assouad dimension} is strictly less than two, see \Cref{def:doubling}. There are quasicircles for every Assouad dimension between $1$ and $\infty$ since $Z = ( \mathbb{S}^{1}, \norm{ \cdot }_{2}^{\alpha} )$ for $0 < \alpha \leq 1$ has Assouad dimension $\alpha^{-1}$.
\begin{prop}\label{lemm:assouad}
Let $Y$ be a quasiconformal Jordan domain. If $\partial Y$ is a quasicircle and the boundary points satisfy the area growth \eqref{eq:area:growth:intro}, then the Assouad dimension of $\partial Y$ is at most two.
\end{prop}
It is not clear if $\partial Y$ in the above statement must be planar. However, if $Y$ is \emph{annularly linearly locally connected (ALLC)} and \emph{Ahlfors $2$-regular}, then $\partial Y$ is a planar quasicircle \cite[Theorems 8.1 and 8.2]{Mer:Wil:13}; see \cite{Mer:Wil:13} for the proofs and terminology. Quasiconformal Jordan domains satisfying these stronger assumptions appear in \cite{Wil:08} and \cite{Bo:Hei:Roh:01}.

We localize these assumptions in the following statement and obtain the same conclusion.
\begin{thm}\label{prop:planar:QC}
Let $Y$ be a quasiconformal Jordan domain such that $\partial Y$ is a quasicircle and its boundary points satisfy the area growth \eqref{eq:area:growth:intro}. Then the Assouad dimension of $\partial Y$ is strictly less than two if the following two conditions are satisfied for some $r_{0} > 0$, $C > 0$ and $\lambda > 1$:
\begin{enumerate}[label=(\alph*)]
    \item\label{prop:planar:QC.1} For every $y \in \partial Y$ and $0 < 2r < R < r_{0}$ and any pair $a, b \in \overline{B}_{ \overline{Y} }( y, R ) \setminus B_{ \overline{Y} }( y, r )$, there exists a path $\abs{ \alpha } \subset \overline{ B }_{ \overline{Y} }( y, \lambda R ) \setminus B_{ \overline{Y} }( y, \lambda^{-1}r )$ containing $a$ and $b$.
    \item\label{prop:planar:QC.2} For every $z \in Y$ with $0 < r < d( z, \partial Y ) \leq r_{0}$, $\mathcal{H}^{2}_{ \overline{Y} }( \overline{ B }_{ \overline{Y} }( z, r ) ) \geq C^{-1} r^{2}$.
\end{enumerate}
In particular, $\partial Y$ is planar.
\end{thm}
If \ref{prop:planar:QC.1} holds we say that $\partial Y$ is \emph{relatively ALLC} and if \ref{prop:planar:QC.2} holds, we say that $Y$ satisfies the \emph{Ahlfors lower bound} near $\partial Y$. The main point of \Cref{prop:planar:QC} is to only restrict the geometry of $\overline{Y}$ near the boundary $\partial Y$.

The relative ALLC guarantees that $\partial Y$ is \emph{porous} in $\overline{Y}$ below the given scale $r_{0}$ (\Cref{eq:boundary:porosity}). The porosity allows us to pack many balls in $Y$, well-disjoint from $\partial Y$, near all points of $\partial Y$ at all scales below $r_{0}$. Now the Ahlfors lower bound, valid for such balls, combined with the upper bound \eqref{eq:area:growth:intro} allow us to control quantitatively the total amount of such non-overlapping balls in a given interval of scales. This quantification allows us to prove planarity for $\partial Y$. This idea appears in \cite[Lemma 3.12]{Bo:Hei:Roh:01}, where the authors prove that a compact set in an Ahlfors regular space is porous if and only if its Assouad dimension is strictly smaller than the homogeneous dimension of the space. A similar argument also works in the setting of \Cref{prop:planar:QC}.

\begin{comment}
\begin{prop}\label{prop:qs:boundary}
Let $\mathcal{C}$ be any quasicircle for which there exists a quasisymmetry $g \colon \mathbb{S}^{1} \rightarrow \mathcal{C}$ with $d_{\mathcal{C}}( g(x), g(y) ) \leq C_{0}\norm{ x - y }_{2}^{ \alpha }$ for some $2^{-1} < \alpha \leq 1$ and $C_{0} > 0$.

Then there exists a quasiconformal Jordan domain $Y$ with the following properties.
\begin{enumerate}[label=(\alph*)]
    \item There exists a bi-Lipschitz map $h \colon \partial Y \rightarrow \mathcal{C}$;
    \item $Y$ satisfies \ref{prop:planar:QC.1} and \ref{prop:planar:QC.2} in \Cref{prop:planar:QC};
    \item there exists a quasisymmetry $\Phi' \colon \overline{ \mathbb{D} } \rightarrow \overline{ Y }$ such that $g = h \circ \Phi'|_{ \mathbb{S}^{1} }$ and $\Phi'|_{ \mathbb{D} }$ is quasiconformal;
    \item every quasiconformal homeomorphism $\phi \colon \mathbb{D} \rightarrow Y$ extends to a quasisymmetry $\Phi \colon \overline{ \mathbb{D} } \rightarrow \overline{ Y }$;
    \item any quasisymmetry $g \colon \mathbb{S}^{1} \rightarrow \partial Y$ extends to a quasiconformal homeomorphism $\phi \colon \mathbb{D} \rightarrow Y$.
\end{enumerate}
\end{prop}
Observe that $1/\alpha < 2$ is an upper bound to the Assouad dimension of $\mathcal{C}$. The proof of \Cref{prop:qs:boundary} breaks down when $\alpha = 2^{-1}$ and $g(x) = x$ is the snowflake map $g \colon \mathbb{S}^{1} \rightarrow ( \mathbb{S}^{1}, \norm{ \cdot }_{2}^{1/2})$.
\end{comment}

\subsection{Outline}
In \Cref{sec:preliminaries}, we introduce the notations we use and some preliminary results. In \Cref{sec:car}, we prove \Cref{thm:carat:metric}. \Cref{thm:QS} and \Cref{lemm:assouad} are proved in \Cref{sec:BA}. \Cref{prop:planar:QC} is proved in \Cref{sec:PQ}. \Cref{sec:conc} contains some concluding remarks.

\section{Preliminaries}\label{sec:preliminaries}

\subsection{Notation} \label{sec:notation} 

    Let $(Y,d_{Y})$ be a metric space. The open ball centered at a point $y \in Y$ of radius $r>0$ with respect to the metric $d$ is denoted by $B_{Y}(y,r)$. The closed ball is denoted by $\overline{B}_{ Y }( y, r )$. We sometimes omit the subscript from $d_{Y}$, from $B_{Y}$, and from $\overline{B}_{Y}$, respectively.

    We recall the definition of Hausdorff measure. Let $(Y,d)$ be a metric space. For all $Q \geq 0$, the \emph{$Q$-dimensional Hausdorff measure (Hausdorff $Q$-measure)} is defined by
	\[
	    \mathcal{H}_Y^{Q}(B)
	    =
	    \frac{\alpha(Q)}{2^Q}
	    \sup_{ \delta > 0 }
	    \inf
	        \left\{
	            \sum_{i=1}^\infty (\diam B_i)^Q
	            \colon
	            B \subset \bigcup_{i=1}^\infty B_i, 
	            \diam B_i < \delta
	        \right\} 
	\]
	for all sets $B \subset Y$, where $\alpha(Q) = \pi^{ \frac{Q}{2} } \left( \Gamma\left( Q/2 + 1 \right) \right)^{-1}$. The constant $\alpha(Q)$ is chosen in such a way that $\mathcal{H}^{n}_{\mathbb{R}^{n}}$ coincides with the Lebesgue measure $\mathcal{L}^{n}$ for all positive integers.
    
    The \emph{length} of a path $\gamma\colon [a,b] \to Y$ is defined as 
    \[ 
        \ell_d(\gamma) = \sup \sum_{i=1}^n d(\gamma(t_{i-1}), \gamma(t_i)),
    \]
    the supremum taken over all finite partitions $a = t_0 \leq t_1 \leq \cdots \leq t_n = b$. A path is \emph{rectifiable} if it has finite length.
    
    The \emph{metric speed} of a path $\gamma \colon \left[a,  b\right] \rightarrow Y$ at the point $t \in \left[a,  b\right]$ is defined as
    \begin{equation*}
    \label{eq:metric:speed:path}
        v_{\gamma}(t)
        =
        \lim_{ h \rightarrow 0^{+} }
            \frac{ d( \gamma( t + h ),  \gamma( t ) ) }{ h }
    \end{equation*}
    whenever this limit exists. If $\gamma$ is rectifiable, its metric speed exists at $\mathcal{L}^{1}$-almost every $t \in \left[a,  b\right]$ \cite[Theorem 2.1]{Dud:07}.

    A rectifiable path $\gamma\colon [a,b] \to Y$ is \emph{absolutely continuous} if for all $a \leq s \leq t \leq b$,
    \begin{equation*}
    \label{eq:AC:characterization}
        d( \gamma(t),  \gamma(s) )
        \leq
        \int_{ s }^{ t }
            v_{ \gamma }( u )
        \,d\mathcal{L}^{1}( u )
    \end{equation*}
    with $v_{ \gamma } \in L^{1}( \left[a,  b\right] )$ and $\mathcal{L}^{1}$ the Lebesgue measure on the real line. Equivalently, $\gamma$ is absolutely continuous if it maps sets of $\mathcal{L}^{1}$-measure zero to sets of $\mathcal{H}_Y^{1}$-measure zero in its image \cite[Section 3]{Dud:07}.
    
    Let $\gamma \colon \left[a, b\right] \rightarrow X$ be an absolutely continuous path. Then the \emph{path integral} of a Borel function $\rho \colon X \rightarrow \left[0,  \infty\right]$ \emph{over $\gamma$} is
    \begin{equation}
    \label{eq:path:integral}
        \int_{ \gamma }
            \rho
        \,ds
        =
        \int_{ a }^{ b }
            ( \rho \circ \gamma )
            v_{ \gamma }
        \,d\mathcal{L}^{1}.
    \end{equation}
    If $\gamma$ is rectifiable, then the \emph{path integral} of $\rho$ over $\gamma$ is defined to be the path integral of $\rho$ over the arc length parametrization $\gamma_{s}$ of $\gamma$; see for example Chapter 5 of \cite{HKST:15}.

    \subsection{Quasiconformal Jordan domains}
    We assume that $Y$ is a quasiconformal Jordan domain. In particular, its completion $\overline{Y}$ is homeomorphic to $\left[0, 1\right]^{2}$ and has finite Hausdorff $2$-measure.
    
    Given a Borel set $A \subset Y$, the \emph{length} of a path $\gamma \colon \left[a, b\right] \rightarrow Y$ \emph{in $A$} is defined as $\int_{ Y } \chi_{A}(y) \#( \gamma^{-1}(y) ) \,d\mathcal{H}^{1}_{Y}(y)$, where $\#( \gamma^{-1}(x) )$ is the counting measure of $\gamma^{-1}(x)$. This formula makes sense for paths that are not necessarily rectifiable \cite[Theorem 2.10.13]{Fed:69}. When $\gamma$ is rectifiable, for every Borel function $\rho \colon \overline{Y} \rightarrow \left[0, \infty\right]$,
    \begin{equation}
        \label{eq:areaformula}
        \int_{ \gamma } \rho \,ds
        =
        \int_{ \overline{Y} }
            \rho( x )
            \#( \gamma^{-1}(x) )
        \,d\mathcal{H}^{1}_{\overline{Y}}(x).
    \end{equation}
    The equality \eqref{eq:areaformula} follows from \cite[Theorem 2.10.13]{Fed:69} via a standard approximation argument using simple functions.

    We recall a special case of the coarea inequality \cite[2.10.25]{Fed:69}. Let $\alpha \in \left\{ 0, 1\right\}$, $B \subset \overline{Y}$ Borel, and $f \colon \overline{Y} \rightarrow \mathbb{R}$ $1$-Lipschitz. Then
    \begin{equation}
        \label{eq:eilenberg}
        \int_{ \mathbb{R} }^{*}
            \mathcal{H}^{\alpha}_{\overline{Y}}( B \cap f^{-1}(t) )
        \,d\mathcal{L}^{1}(t)
        \leq
        C_{\alpha}
        \mathcal{H}^{\alpha+1}_{ \overline{Y} }( B ),
    \end{equation}
    where $C_{0} = 1$ and $C_{1} = 4/\pi$. Here $\int^{*}$ refers to the upper integral \cite[2.4.2]{Fed:69}. If $\mathcal{H}^{\alpha+1}_{\overline{Y}}( B ) < \infty$, the upper integral can be replaced with the usual one \cite[2.10.26]{Fed:69}.
    
    Via a standard approximation argument using simple functions, we obtain the following.
    \begin{thm}\label{thm:eilenberg}
    Let $f \colon \overline{Y} \rightarrow \mathbb{R}$ be $1$-Lipschitz, $\alpha \in \left\{0, 1\right\}$ and $C_{\alpha}$ as in \eqref{eq:eilenberg}. Then, for every Borel function $g \colon \overline{Y} \rightarrow \left[0, \infty\right]$,
    \begin{equation*}
        \int_{ \mathbb{R} }^{*}
        \int_{ f^{-1}(t) }
            g(y)
        \,d\mathcal{H}^{\alpha}_{\overline{Y}}(y)
        \,d\mathcal{L}^{1}(t)
        \leq
        C_{\alpha}
        \int_{ \overline{Y} }
            g(y)
        \,d\mathcal{H}^{\alpha+1}_{\overline{Y}}(y).
    \end{equation*}
    When $g$ is $\mathcal{H}^{\alpha+1}_{\overline{Y}}$-integrable, the upper integral can be replaced with the usual one.
    \end{thm}
    
    In the following, we say that $C \subset \overline{Y}$ is a \emph{continuum} if $C$ is compact and connected. A compact set $F \subset \overline{Y}$ \emph{separates} $x, y  \in \overline{Y}$ if $x, y \in \overline{Y}  \setminus F$ and the points are in different connected components of $\overline{Y} \setminus F$.

    \begin{lemm}\label{thm:topology}
    Let $F \subset \overline{Y}$ be compact and $x, y \in \overline{Y}$ separated by $F$. Then there exists a continuum $C \subset F$ that separates $x$ and $y$.
    \end{lemm}
    \begin{proof}
    Let $x$, $y$ and $F$ be as in the claim. We consider a homeomorphism $h \colon \overline{Y} \rightarrow Z \subset \mathbb{S}^{2}$, where $Z$ is the union of the equator and the southern hemisphere of $\mathbb{S}^{2}$. Then there exists a nested sequence of quadrilaterals $Z_{n} \supset Z_{n+1} \supset Z$ such that $h(x), h(y) \in Z_{n}$ is an interior point of $Z_{n}$ for each $n \in \mathbb{N}$ and $Z = \bigcap_{ n = 1 }^{ \infty } Z_{n}$.
    
    Since $F_{n} = \partial Z_{n} \cup h(F) \subset Z_{n}$ separates $h(x)$ and $h(y)$ in $\mathbb{S}^{2}$, there exists a continuum $C_{n} \subset F_{n}$ separating $h(x)$ and $h(y)$ in $\mathbb{S}^{2}$ \cite[Chapter 2, Lemma 5.20]{Wil:79}. In particular, for every path $\gamma \colon \left[0, 1\right] \rightarrow Z$ joining $h(x)$ to $h(y)$, there exists $z_{n} \in C_{n} \cap |\gamma|$ for every $n \in \mathbb{N}$. Up to passing to a subsequence and relabeling, the continua $( C_{n} )_{n=1}^{\infty}$ converge to a continuum $C' \subset \bigcap_{ n = 1 }^{ \infty } h(F) \cap Z_{n} = h(F)$ in the Hausdorff convergence \cite[Theorems 4.4.15 and 4.4.17]{Amb:Til:04}. If $\gamma$ and $( z_{n} )_{ n = 1 }^{ \infty }$ are as above, the accumulation points of $( z_{n} )_{ n = 1 }^{ \infty }$ are contained in $C' \cap |\gamma|$ \cite[Proposition 4.4.14]{Amb:Til:04}. Consequently, $C' \cap |\gamma| \neq \emptyset$ for every such $\gamma$. Hence $C = h^{-1}( C' ) \subset F$ is a continuum separating $x$ and $y$.
    \end{proof}

\subsection{Metric Sobolev spaces}\label{sec:sobolev}

    In this section we give an overview of Sobolev theory in the metric surface setting, and refer to \cite{HKST:15} for a comprehensive introduction.

    Let $\Gamma \subset \mathcal{C}( \left[0, 1\right]; Y )$ be a family of rectifiable paths in $Y$. A Borel function $\rho\colon Y \to [0, \infty]$ is \emph{admissible} for $\Gamma$ if the path integral $\int_\gamma \rho\,ds \geq 1$ for all rectifiable paths $\gamma \in \Gamma$. The \emph{modulus} of $\Gamma$ is
        \[\Mod \Gamma = \inf \int_Y \rho^2\,d\mathcal{H}_Y^2,\]
    where the infimum is taken over all admissible functions $\rho$. Observe that if $\Gamma_{1}$ and $\Gamma_{2}$ are path families and every path $\gamma_{1} \in \Gamma_{1}$ contains a subpath $\gamma_{2} \in \Gamma_{2}$, then $\Mod \Gamma_{1} \leq \Mod \Gamma_{2}$. In particular, this holds if $\Gamma_{1} \subset \Gamma_{2}$. A property holds for \emph{almost every} path if the family of paths for which the property fails has zero modulus.
    
    Let $\psi \colon (Y,d_Y) \to (Z,d_Z)$ be a mapping between metric spaces $Y$ and $Z$. A Borel function $\rho \colon Y \to [0, \infty]$ is an \emph{upper gradient of $\psi$} if
    \[
        d_Y(\psi(x),\psi(y)) \leq \int_\gamma \rho\,ds
    \]
    for every rectifiable path $\gamma\colon [0,1] \to Y$ connecting $x$ to $y$. The function $\rho$ is a \emph{weak upper gradient of $\psi$} if the same holds for almost every rectifiable path.
    
    A weak upper gradient $\rho \in L^{2}_{\loc}( Y )$  of $\psi$ is \emph{minimal} if it satisfies $\rho \leq \widetilde{\rho}$ almost everywhere for all weak upper gradients $\widetilde{\rho} \in L^{2}_{\loc}( Y )$ of $\psi$. If $\psi$ has a weak upper gradient $\rho \in L^{2}_{\loc}( Y )$, then $\psi$ has a minimal weak upper gradient, which we denote by $\rho_\psi$. We refer to Section 6 of \cite{HKST:15} and Section 3 of \cite{Wil:12} for details.

    Fix a point $z \in Z$, and let $d_z = d_Z(\cdot,z)$. The space $L^{2}( Y,  Z )$ is defined as the collection of measurable maps $\psi \colon Y \to Z$ such that $d_z \circ \psi$ is in $L^2(Y)$.

    Moreover, $L^{2}_{\loc}( Y,  Z )$ is defined as those measurable maps $\psi \colon Y \to Z$ for which, for all $y \in Y$, there is an open set $U \subset Y$ containing $y$ such that $\psi|_U$ is in $L^2(U,Z)$.
    
    The metric Sobolev space $N^{1,  2}_{\loc}( Y,  Z )$ consists of those maps $\psi \colon Y \rightarrow Z$ in $L^{2}_{\loc}( Y,  Z )$ that have a minimal weak upper gradient $\rho_{ \psi } \in L^{2}_{\loc}( Y )$.
    
    For open $\emptyset \neq U \subset Y$, we say that $\psi \in N^{1,  2}( U,  Z )$ if $\psi|_{U} \in N^{1,  2}_{\loc}( U, Z )$, $\rho_{\psi|_U} \in L^{2}( U )$ and $\psi|_{U} \in L^{2}( U, Z )$.

    Given a homeomorphism $\psi \colon Y \rightarrow Z$, the pullback measure $\psi^{*}\mathcal{H}^{2}_{Z}$ is defined by $\psi^{*}\mathcal{H}^{2}_{Z}( B ) = \mathcal{H}^{2}_{Z}( \psi(B) )$ for each Borel set $B \subset Y$. The pullback measure has a decomposition $\psi^{*}\mathcal{H}^{2}_{Z} = J_{\psi} \mathcal{H}^{2}_{Y} + \mu^{\perp}$, where $J_{\psi}$ is locally integrable with respect to $\mathcal{H}^{2}_{Y}$, and the measures $\mathcal{H}^{2}_{Y}$ and $\mu^{\perp}$ are singular \cite[Sections 3.1-3.2 in Volume I]{Bog:07}. We call the density $J_{\psi}$ the \emph{Jacobian} of $\psi$.

\subsection{Quasiconformal mappings} \label{sec:quasiconformal_mappings}
    We define quasiconformal maps and recall some basics.
    \begin{defi}\label{def:QCmaps}
    Let $( Y, d_{Y} )$ and $( Z, d_{Z} )$ be metric spaces with locally finite Hausdorff $2$-measures. We say that a homeomorphism $\psi \colon ( Y, d_{Y} ) \rightarrow ( Z, d_{Z} )$ is \emph{quasiconformal} if
    there exists $K \geq 1$ such that for all path families $\Gamma$ in $Y$
    \begin{equation}
        \label{eq:def:QCmaps}
        K^{-1} \Mod \Gamma
        \leq
        \Mod \psi \Gamma
        \leq
        K \Mod \Gamma,
    \end{equation}
    where $\psi \Gamma = \left\{ \psi \circ \gamma \colon \gamma \in \Gamma \right\}$. If \eqref{eq:def:QCmaps} holds with a constant $K \geq 1$, we say that $\psi$ is \emph{$K$-quasiconformal}.
    \end{defi}
    \Cref{def:QCmaps} is sometimes called the \emph{geometric} definition of quasiconformality. A special case of \cite[Theorem 1.1]{Wil:12} yields the following.
    \begin{thm} \label{prop:williams:L-Wversion}
    Let $Y$ and $Z$ be metric spaces with locally finite Hausdorff $2$-measure and $\psi \colon Y \rightarrow Z$ a homeomorphism. The following are equivalent for the same constant $K > 0$:
    \begin{enumerate}
        \item[(i)] $\Mod \Gamma \leq K \Mod \psi\Gamma$ for all path families $\Gamma$ in $Y$.
        \item[(ii)] $\psi \in N_{\loc}^{1,2}(Y,Z)$ and satisfies $\rho_\psi^2(y) \leq KJ_\psi(y)$ for $\mathcal{H}_Y^2$-almost every $y \in Y$.
    \end{enumerate}
    \end{thm}
    The \emph{outer dilatation} of $\psi$ is the smallest constant $K_{O} \geq 0$ for which the modulus inequality $\Mod \Gamma \leq K_{O} \Mod \psi \Gamma$ holds for all $\Gamma$ in $Y$. The \emph{inner dilatation} of $\psi$ is the smallest constant $K_{I} \geq 0$ for which $\Mod \psi \Gamma \leq K \Mod \Gamma$ holds for all $\Gamma$ in $Y$. The number $K( \psi ) = \max\left\{ K_{I}( \psi ), K_{O}( \psi ) \right\}$ is the \emph{maximal dilatation} of $\psi$.
    
    For a set $G \subset Y$ and disjoint sets $F_1, F_2 \subset G$, let $\Gamma(F_1,F_2; G)$ denote the family of paths that start from $F_{1}$, end in $F_{2}$ and whose images are contained in $G$. A \emph{quadrilateral} is a set $Q$ homeomorphic to $[0,1]^2$ with boundary $\partial Q$ consisting of four boundary arcs, overlapping only at the end points, labelled $\xi_1, \xi_2, \xi_3, \xi_4$ in cyclic order.
    
    \begin{defi}\label{defi:reciprocal}
    A metric surface $Y$ is \emph{reciprocal} if there exists a constant $\kappa  \geq 1 $ such that
    \begin{align}
    	\label{upper:bound}
    	\kappa^{-1}
    	\leq
    	\Mod \Gamma\left( \xi_{1},  \xi_{3};  Q \right)
    	\Mod \Gamma\left( \xi_{2},  \xi_{4};  Q \right)
    	\leq
    	\kappa
    \end{align}
    for every quadrilateral $Q \subset Y$, and 
    \begin{equation}
    	\label{point:zero:modulus}
    	\lim_{ r \rightarrow 0^{+} }
    	\Mod \Gamma\left( \overline{B}_{Y}( y,  r ),  Y \setminus B_{Y}( y,  R );  \overline{B}_{Y}( y,  R ) \right)
    	=
    	0
    \end{equation}
    for all $y \in Y$ and $R > 0$ such that $Y \setminus B_Y( y, R ) \neq \emptyset$.
    \end{defi}
    We note that the product in \eqref{upper:bound} is always bounded from below by a universal constant $\kappa_{0} >0$ \cite{RR:19}. We also have the following.

    \begin{prop}[Corollary 12.3 of \cite{Raj:17}]\label{prop:outer:to:maximal}
    Let $Y$ be a metric surface, $U \subset Y$ a domain, and $\psi \colon U \rightarrow \Omega \subset \mathbb{R}^{2}$ a homeomorphism. If $K_{O}( \psi ) < \infty$, then $K_{I}( \psi ) \leq \left( 2 \cdot \kappa_{0} \right)K_{O}( \psi ) < \infty$.
    \end{prop}
    Recall the definition of quasiconformal Jordan domain from the introduction.
    \begin{prop}\label{prop:pointwisemodulus}
    Let $Y$ be a quasiconformal Jordan domain and $\Psi \colon \overline{Y} \rightarrow \overline{ \mathbb{D} }$ a homeomorphism. Then $K_{O}( \Psi ) \leq K < \infty$ if and only if there exists a constant $C > 0$ such that
    \begin{equation}
        \label{eq:pointwise:modulus}
        \liminf_{ r \rightarrow 0^{+} }
        \Mod \Psi^{-1} \Gamma\left( \overline{B}_{ \overline{\mathbb{D}} }( x,  r ),  \overline{\mathbb{D}} \setminus B_{ \overline{\mathbb{D}} }( x,  2r );  \overline{B}_{ \overline{\mathbb{D}} }( x,  2r ) \right)
        \leq
        C
    \end{equation}
    for every $x \in \overline{ \mathbb{D} }$. The constants $K$ and $C$ depend on each other quantitatively. Moreover, $K_{I}( \Psi ) \leq ( 2 \cdot \kappa_{0} ) \cdot K$.
    \end{prop}
    \begin{proof}

    Since the Lebesgue $2$-measure on $\overline{\mathbb{D}}$ is doubling, Theorem 1.2 of \cite{Wil:12} states that an upper bound $K_{O}( \Psi ) \leq K$ is quantitatively equivalent to the following statement: there exists $C' \geq 1$ such that for every $x \in \overline{ \mathbb{D} }$,
    \begin{equation*}
        \liminf_{ r \rightarrow 0^{+} }
        \frac{
            r^{2}
            \Mod
            \Psi^{-1} \Gamma\left( \overline{B}_{ \overline{\mathbb{D}} }( x,  r ),  \overline{\mathbb{D}} \setminus B_{ \overline{\mathbb{D}} }( x,  2r );  \overline{B}_{ \overline{\mathbb{D}} }( x,  2r ) \right)
        }{
            \mathcal{L}^{2}( \overline{B}_{ \overline{\mathbb{D}} }( x, r ) )
        }
        \leq
        C'.
    \end{equation*}
    Since $\mathcal{L}^{2}( B_{ \overline{\mathbb{D}} }( x, r ) )$ is comparable to $r^{2}$, $K_{O}( \Psi ) \leq K$ if and only if \eqref{eq:pointwise:modulus} holds for some $C$, with $K$ and $C$ depending on one another quantitatively.
    
    It remains to prove that $K_{I}( \Psi ) \leq C_{0} K$. To this end, consider $\psi = \Psi|_{Y}$ and $\phi = \psi^{-1}$. \Cref{prop:outer:to:maximal} implies that $K_{I}( \psi ) = K_{O}( \phi ) \leq C_{0}K$. Then \Cref{prop:williams:L-Wversion} implies that $\phi \in N^{1, 2}( \mathbb{D}; Y )$ since the Jacobian $J_{\phi}$ of $\phi$ is integrable. Observe that the extension $\Phi = \Psi^{-1}$ of $\phi$ is an element of $N^{1, 2}( \overline{ \mathbb{D} }; \overline{Y} )$. This can be seen by extending $\Phi$ to a neighbourhood of $\overline{ \mathbb{D} }$ via reflection over $\mathbb{S}^{1}$ and by applying \cite[Theorem 1.12.3]{Kor:Sch:93}. The minimal weak upper gradient of $\Phi$ has a representative that vanishes in $\mathbb{S}^{1}$ since $\mathcal{L}^{2}( \mathbb{S}^{1} ) = 0$. Therefore $\rho_{ \Phi }^{2} \leq C_{0}K J_{ \Phi }$ holds $\mathcal{L}^{2}$-almost everywhere in $\overline{ \mathbb{D} }$. This implies that $\Phi$ satisfies the second condition in \Cref{prop:williams:L-Wversion} with the constant $C_{0}K$.
    \end{proof}

    We recall a sufficient condition for \eqref{eq:pointshavezeromod} for later use.
    \begin{lemm}\label{lemm:points}
    Suppose that there exists $C_{U} > 0$ such that for all $y \in \partial Y$ and $0 < r < \diam \partial Y$,
    \begin{equation}
        \label{eq:area:growth}
        \mathcal{H}^{2}_{ \overline{Y} }( \overline{B}_{ \overline{Y} }( y, r ) )
        \leq
        C_{U} r^{2}.
    \end{equation}
    Then for $\widetilde{C}_{U} = 8C_{U} / \log 2$ and for every $y \in \partial Y$ and $0 < 2 r < R < 2^{-1}\diam \partial Y$,
    \begin{equation}
        \label{eq:pointshavezeromod:quantitative}
        \Mod
        \Gamma( 
            \overline{B}_{ \overline{Y} }( y, r ),
            \overline{Y} \setminus B_{ \overline{Y} }( y, R );
            \overline{Y}
        )
        \leq
        \frac{ \widetilde{C}_{U} }{ \log \frac{R}{r} }.
    \end{equation}
    In particular, \eqref{eq:pointshavezeromod} in \Cref{prop:QS} holds under the assumption \eqref{eq:area:growth}.
    \end{lemm}
    
    \begin{proof}

    The inequality \eqref{eq:pointshavezeromod:quantitative} follows from \eqref{eq:area:growth} by considering the admissible function $\rho(x) = \frac{ 1 }{ \log \frac{R}{r} } \frac{ 1 }{ d( y, x ) } \chi_{ \left\{ r \leq d(y,x) \leq R \right\} }$. We claim that $\rho$ is admissible for the family $\Gamma( \overline{B}_{ \overline{Y} }( y, r ), \overline{Y} \setminus B_{ \overline{Y} }( y, R ); \overline{Y} )$. To this end, fix a rectifiable $\gamma \in \Gamma( \overline{B}_{ \overline{Y} }( y, r ), \overline{Y} \setminus B_{ \overline{Y} }( y, R ); \overline{Y} )$.
    
    We denote $f(x) = d(y,x)$. Whenever $x \in f^{-1}(t)$, we have $\rho( x ) {\#}( \gamma^{-1}(x) ) \geq \frac{ 1 }{ \log \frac{R}{r} } \frac{ 1 }{ t } \chi_{ \left\{ r \leq t \leq R \right\} }$. Therefore
    \begin{equation*}
        \int_{ \mathbb{R} }^{*}
        \int_{ f^{-1}(t) }
            \rho(x) {\#}( \gamma^{-1}(x) )
        \,d\mathcal{H}^{0}_{ \overline{Y} }(x)
        \,d\mathcal{L}^{1}(t)
        \geq
        \int_{ r }^{ R }
            \frac{ 1 }{ \log \frac{R}{r} }
            \frac{ 1 }{ t }
        \,d\mathcal{L}^{1}(t)
        =
        1.
    \end{equation*}
    Then \Cref{thm:eilenberg} implies
    \begin{equation*}
        \int_{ \overline{Y} }
            \rho( x )
            \#( \gamma^{-1}(x) )
        \,d\mathcal{H}^{1}_{\overline{Y}}(x)
        \geq
        \int_{ \mathbb{R} }^{*}
        \int_{ f^{-1}(t) }
            \rho(x) {\#}( \gamma^{-1}(x) )
        \,d\mathcal{H}^{0}_{ \overline{Y} }(x)
        \,d\mathcal{L}^{1}(t).
    \end{equation*}
    The equality \eqref{eq:areaformula} yields $\int_{ \gamma } \rho \,ds
        =
        \int_{ \overline{Y} }
            \rho( x )
            \#( \gamma^{-1}(x) )
        \,d\mathcal{H}^{1}_{\overline{Y}}(x)$. Hence $\rho$ is admissible for $\Gamma( \overline{B}_{ \overline{Y} }( y, r ), \overline{Y} \setminus B_{ \overline{Y} }( y, R ); \overline{Y} )$.
    
    The $L^{2}$-norm of $\rho$ is estimated from above by applying the area growth \eqref{eq:area:growth} on the annuli $A_{l} = \left\{ 2^{l} r \leq d( y, x ) < 2^{l+1} r \right\}$ for $l = 0,1, 2, \dots, k$ for $2^{k} r < R \leq 2^{k+1} r$, $k \in \mathbb{N}$. That is,
    \begin{align*}
        \int_{ \overline{Y} }
            \rho^{2}(x)
        \,d\mathcal{H}^{2}_{ \overline{Y} }(x)
        &\leq
        \sum_{ l = 0 }^{ k }
        \int_{ A_{l} }
            \rho^{2}(x)
        \,d\mathcal{H}^{2}_{ \overline{Y} }(x)
        \leq
        \frac{ 1 }{ \log^{2}( \frac{R}{r} ) }
        \sum_{ l = 0 }^{ k }
            \frac{  \mathcal{H}^{2}_{ \overline{Y} }( \overline{B}( y, 2^{l+1}r ) ) }
            { 2^{ 2l  }r^{2} }
        \\
        &\leq
        \frac{ 1 }{ \log^{2}( \frac{R}{r} ) }
        \sum_{ l = 0 }^{ k }
            \frac{  C_{U} 2^{ 2l +2 } r^{2} }
            { 2^{ 2l } r^{2} }
        =
        4C_{U}
        \frac{ k+1 }{ \log^{2} \frac{ R }{ r }  }
        \leq
        \frac{ 8C_{U} / \log 2 }{ \log \frac{R}{r} }
    \end{align*}
    since $k+1 \leq ( 2 / \log 2 ) \log \frac{R}{r}$. The inequality \eqref{eq:pointshavezeromod:quantitative} follows.
    
    We claim now that \eqref{eq:pointshavezeromod} in \Cref{prop:QS} holds. Let $y \in \partial Y$ and $R' > R > 0$ such that $\overline{Y} \setminus B_{ \overline{Y} }( y, R' ) \neq \emptyset$ and $2^{-1} \diam \partial Y > R$. Then for every $0 < 2 r < R$, every path in $\Gamma( \overline{B}_{ \overline{Y} }( y, r ), \overline{Y} \setminus B_{ \overline{Y} }( y, R' ); \overline{Y} )$ has a subpath in $\Gamma( \overline{B}_{ \overline{Y} }( y, r ), \overline{Y} \setminus B_{ \overline{Y} }( y, R ); \overline{Y} )$. Hence
    \begin{align*}
        \Mod
        \Gamma( 
            \overline{B}_{ \overline{Y} }( y, r ),
            \overline{Y} \setminus B_{ \overline{Y} }( y, R' );
            \overline{Y}
        )
        \leq
        \Mod
        \Gamma( 
            \overline{B}_{ \overline{Y} }( y, r ),
            \overline{Y} \setminus B_{ \overline{Y} }( y, R );
            \overline{Y}
        ).
    \end{align*}
    The right-hand side converges to zero as $r \rightarrow 0^{+}$, given \eqref{eq:pointshavezeromod:quantitative}. This establishes \eqref{eq:pointshavezeromod}.
    
    \end{proof}

\subsection{Quasicircles}\label{sec:quasicircle}
In this section we recall some basic properties of quasisymmetries and quasicircles. If $g \colon ( Y, d_{Y} ) \rightarrow ( Z, d_{Z} )$, we denote
\begin{align*}
    \label{eq:L}
    L_{g}( y, r )
    &=
    \sup_{ w \in \overline{B}_{ Y }( y, r ) }
    d_{Z}( g(y), g(w) )
%    \text{ and}
    \quad \text{and} \quad
%    \\
%    \label{eq:ell}
    \ell_{g}( y, r )
%    &=
    =
    \inf_{ w \in Y \setminus B_{Y}( y, r ) }
    d_{Z}( g(y), g(w) ).
\end{align*}
\begin{defi}\label{def:quasisymmetry}
Let $\eta \colon \left[0, \infty\right) \rightarrow \left[0, \infty\right)$ be a homeomorphism. A homeomorphism $g \colon ( Y, d_{Y} ) \rightarrow ( Z, d_{Z} )$ between metric spaces is \emph{$\eta$-quasisymmetric} if for every $y \in Y$ and $0 < r_{1}, r_{2} < \diam Y$,
\begin{equation}
    \label{eq:distortion}
    L_{g}( y, r_{1} )
    \leq
    \eta\left( \frac{ r_{1} }{ r_{2} } \right)
    \ell_{g}( y, r_{2} ).
\end{equation}
A homeomorphism $g$ is \emph{quasisymmetric} if it is $\eta$-quasisymmetric for some homeomorphism $\eta \colon \left[0, \infty\right) \rightarrow \left[0, \infty\right)$.
\end{defi}
A set $S \subset Y$ is \emph{$r$-separated} if for every $x, y \in S$ with $x \neq y$, $d_{Y}( x, y ) \geq r$, and an \emph{$r$-net} if for every $y \in Y$, there exists $x \in S$ for which $d_{Y}( x, y ) < r$. An $r$-separated set is \emph{maximal} if it is also an $r$-net.

\begin{defi}\label{def:doubling}
A metric space $( Y, d_{Y} )$ has its \emph{Assouad dimension} bounded from above by $Q > 0$ if for every $0 < \epsilon < 1$ and every $( y, r ) \in Y \times ( 0, \diam Y )$, any $\epsilon r$-separated set $S \subset B_{Y}( y, r )$ satisfies
\begin{equation}
    \label{eq:cardinality}
    {\#}S
    \leq
    C \epsilon^{-Q},
\end{equation}
where $C$ is a constant independent of $\epsilon$, $y$, $r$ and $S$. Here ${\#}S$ refers to the counting measure of $S$. The \emph{Assouad dimension} of $Y$ is the infimum of such $Q$.
\end{defi}
A metric space $( Y, d_{Y} )$ is said to be \emph{doubling} if its Assouad dimension is finite.

\begin{defi}\label{def:boundedturning}
Let $\lambda \geq 1$. A metric space $( Y, d_{Y} )$ has \emph{$\lambda$-bounded turning} if for every $y, z \in Y$ there exists a compact and connected set $E \subset Y$ containing $y$ and $z$ such that $\diam E \leq \lambda d_{Y}( y, z )$.
\end{defi}

We recall that a metric space $\mathcal{C}$ homeomorphic to $\mathbb{S}^{1}$ is a quasisymmetric image of $\mathbb{S}^{1}$ if and only if $\mathcal{C}$ has bounded turning and is doubling \cite{Tuk:Vai:80}. We refer to any quasisymmetric image of $\mathbb{S}^{1}$ as a \emph{quasicircle}.

\section{Carathéodory's theorem}\label{sec:car}

\subsection{Proof of \Cref{thm:carat:metric}}
We fix a quasiconformal homeomorphism $\phi \colon \mathbb{D} \rightarrow Y$ and claim that it has a monotone and surjective extension $\Phi \colon \overline{ \mathbb{D} } \rightarrow \overline{Y}$. We are assuming that $\overline{Y}$ is homeomorphic to $\left[0, 1\right]^{2}$ and has finite Hausdorff $2$-measure.

Fix $x_{0} \in \mathbb{S}^{1}$. For each $0 < r < 2^{-1}$, denote $E_{r} = \mathbb{S}^{1}( x_{0}, r ) \cap \overline{ \mathbb{D} }$. Let $V_{r}$ be the component of $\mathbb{D} \setminus E_r$ whose closure contains $\left\{ x_{0} \right\}$.

Let $U_{r} = \overline{ \phi( V_{r} ) }$. Since $V_{r}$ is connected, so is $U_{r}$. Moreover, as $V_{r'} \subset V_{r}$ whenever $r' < r$, we have $U_{r'} \subset U_{r}$. Therefore
\begin{equation}
    \label{eq:intersection:end}
    \emptyset \neq \widetilde{C} = \bigcap_{ 0 < r < 2^{-1} } U_{r}
    \quad\text{is compact and connected \cite[Theorem 28.2]{Wil:70}}.
\end{equation}
Notice that $\widetilde{C} \subset \partial Y$.

\Cref{lemm:intersection:point} implies that $\widetilde{C}$ is a singleton. Let $y_{0}$ denote the unique element. We define $\Phi( x_{0} ) \coloneqq y_{0}$. We repeat the argument for every $x_{0} \in \mathbb{S}^{1}$. By setting $\Phi(x) = \phi(x)$ for every $x \in \mathbb{D}$, we obtain a mapping
\begin{equation*}
%    \label{eq:extension}
    \Phi \colon \overline{ \mathbb{D} } \rightarrow \overline{Y}.
\end{equation*}
We prove in \Cref{lemm:continuous} that $\Phi$ is continuous and surjective. \Cref{lemm:surjectivemonotone} shows the monotonicity of $\Phi$. Hence \Cref{thm:carat:metric} follows after we verify these lemmas.

\begin{lemm}\label{lemm:intersection:point}
Let $d_{r}$ denote the diameter of $U_{r}$. Then $d_{r} \rightarrow \diam \widetilde{C} = 0$ for every $x_{0} \in \mathbb{S}^{1}$.
\end{lemm}

Before proving \Cref{lemm:intersection:point}, we prove a couple of technical lemmas. In the following, an \emph{arc} refers to a set homeomorphic to $\left[0, 1\right]$.

\begin{lemm}\label{lemm:positivemodulus}
Let $C' \subset Y$ be an arc and $C \subset \partial Y$ a compact and connected set. Then
\begin{equation}
    \label{eq:intersection:nontrivial}
    \diam C > 0
    \quad\text{implies}\quad
    \Mod \Gamma( C, C'; Y \cup C )
    >
    0.
\end{equation}
\end{lemm}
\begin{proof}[Proof of \Cref{lemm:positivemodulus}]
Since $C$ and $C'$ are disjoint, there are Borel functions $\rho \in L^{2}( \overline{Y} )$ admissible for $\Gamma( C, C'; Y \cup C )$. We fix such a function $\rho$ and find a lower bound for the $L^{2}$-norm of $\rho$, depending only on $C$ and $C'$. The claim \eqref{eq:intersection:nontrivial} follows from this.

We argue as in the proof of \cite[Proposition 3.5]{Raj:17}. First, we join $C$ and $C'$ with an arc $\gamma \colon \left[0, 1\right] \rightarrow Y \cup C$ for which $r_{0} = d( \abs{\gamma}, \partial Y \setminus C ) > 0$, and consider the Lipschitz function $f(z) = d( \abs{\gamma}, z )$. Since $C$ and $C'$ are arcs, we can choose $\gamma$ in such a way that $\gamma(0)$ separates $C$ into two arcs $J_{1}$ and $J_{2}$, $\gamma(1)$ separates $C'$ into two arcs $J_{3}$ and $J_{4}$, and $\gamma(t) \in Y \setminus ( C \cup C' )$ for every $0 < t < 1$.

Fix $0 < r_{1} < r_{0}$ such that every $J_{i}$ intersects $f^{-1}( r )$ for every $0 < r < r_{1}$. For every such $r > 0$, the level set $f^{-1}( r )$ separates $\abs{ \gamma }$ from the arc $\overline{ \partial Y \setminus C }$. Then \Cref{thm:topology} provides us with a continuum $\Gamma_{r} \subset f^{-1}( r )$ that separates $\abs{ \gamma }$ from $\overline{ \partial Y \setminus C }$. The continuum $\Gamma_{r}$ must intersect every $J_{i}$, since otherwise we find a path joining $\abs{ \gamma }$ to $\overline{ \partial Y \setminus C }$ that does not intersect $\Gamma_{r}$.

By applying \Cref{thm:eilenberg} to the function $g(y) = \chi_{ \overline{Y} }(y)$, we conclude that the level set $f^{-1}( r )$ has finite Hausdorff $1$-measure for $\mathcal{L}^{1}$-almost every $0 < r < r_{1}$. In particular, the continuum $\Gamma_{r}$ has finite Hausdorff $1$-measure. Then every pair of points from $\Gamma_{r}$ can be joined with a rectifiable path within $\Gamma_{r}$ \cite[Proposition 15.1]{Sem:96:PI}. Consequently, there exists a rectifiable arc $\theta \colon \left[0, 1\right] \rightarrow \Gamma_{r}$ joining $J_{1} \subset C$ to $J_{3} \subset C'$. Since $0 < r < r_{1} < r_{0}$, we have $\theta \in \Gamma( C, C'; Y \cup C )$. Hence
\begin{equation*}
    1
    \leq
    \int_{ \theta }
        \rho
    \,ds
    \leq
    \int_{ f^{-1}( r ) }
        \rho
    \,d\mathcal{H}^{1}_{ \overline{Y} }.
\end{equation*}
Then \Cref{thm:eilenberg} and Hölder's inequality imply
\begin{equation*}
    r_{1}
    \leq
    \frac{ 4 }{ \pi } \left( \mathcal{H}^{2}_{Y}( \overline{Y} ) \right)^{1/2} \norm{ \rho }_{ L^{2}( \overline{Y} ) }.
\end{equation*}
Rearranging this inequality establishes the claim.
\end{proof}

\begin{lemm}\label{lemm:ACpath}
Let $\theta \colon ( 0, 1 ) \rightarrow G \subset \overline{Y}$ be a homeomorphism and suppose that for every $0 < s < t < 1$,
\begin{equation*}
    \ell( \theta|_{ \left[s, t\right] } )
    \leq
    \int_{ s }^{ t }
        h(a)
    \,d\mathcal{L}^{1}(a)
\end{equation*}
for some $h \in L^{1}( \left[0, 1\right] )$. Then there exists an absolutely continuous extension $\overline{ \theta } \colon \left[0, 1\right] \rightarrow \overline{G}$ of $\theta$ that is surjective.
\end{lemm}
\begin{proof}[Proof of \Cref{lemm:ACpath}]
Let $0 < s < t < 1$. Then
\begin{equation}
    \label{eq:lemm:ACpath:ends}
    d( \theta( s ), \theta( t ) )
    \leq
    \int_{ 0 }^{ 1 }
        \left( \chi_{ \left[0, s\right] }(a) + \chi_{ \left[0, t\right] }(a) \right)
        h(a)
    \,d\mathcal{L}^{1}(a).
\end{equation}
By the absolute continuity of the integral, given $\epsilon > 0$, there exists $\delta > 0$ for which
\begin{equation}
    \label{eq:lemm:ACpath}
    \abs{s}, \abs{t} < \delta
    \quad\text{implies}\quad
    \int_{ 0 }^{ 1 }
        \left( \chi_{ \left[0, s\right] } + \chi_{ \left[0, t\right] } \right)(a)
        h(a)
    \,d\mathcal{L}^{1}(a)
    <
    \epsilon.
\end{equation}
This fact and \eqref{eq:lemm:ACpath:ends} imply that for any given $( s_{j} )_{ j = 1 }^{ \infty } \subset \left( 0, 1 \right)$ converging to zero, the sequence $( \theta( s_{j} ) )_{ j = 1 }^{ \infty }$ is Cauchy. Since $\overline{Y}$ is complete, the sequence converges to some $y_{0} \in \overline{G}$. We define $\overline{ \theta }( 0 ) \coloneqq y_{0}$. The inequalities \eqref{eq:lemm:ACpath:ends} and \eqref{eq:lemm:ACpath} imply that $y_{0}$ is independent of the sequence $( s_{j} )_{ j = 1 }^{ \infty }$, and setting $\overline{\theta}(s) = \theta(s)$ for $0 < s$ defines a continuous extension of $\theta$ to $\left[0, 1\right)$.

By arguing similarly for $t = 1$, we find a continuous extension $\overline{\theta} \colon \left[0, 1\right] \rightarrow \overline{G}$ of $\theta$. The inequality \eqref{eq:lemm:ACpath:ends} extends to every $0 \leq s < t \leq 1$, implying the absolute continuity of $\overline{\theta}$. Notice that for every $y \in \overline{G}$, there exists a sequence $( t_{j} )_{ j = 1 }^{ \infty } \subset ( 0, 1 )$ such that $\overline{\theta}( t_{j} ) \rightarrow y$. By passing to a subsequence and relabeling, we may assume that $( t_{j} )_{ j = 1 }^{ \infty }$ has a limit in $\left[0, 1\right]$. This implies that $\overline{\theta}$ is surjective.
\end{proof}

\begin{proof}[Proof of \Cref{lemm:intersection:point}]
Since $\widetilde{C}$ is the intersection of the $U_{r}$ and $U_{r}$ are nested, we have $d_{r} \rightarrow \diam \widetilde{C}$. Hence the difficulty lies in proving $\diam \widetilde{C} = 0$.

Fix an arc $C' \subset Y$ for which $\phi^{-1}( C' ) \subset \mathbb{D} \setminus ( E_{r} \cup V_{r} )$ for every $0 < r < 2^{-1}$. We assume $\diam \widetilde{C} > 0$ and derive a contradiction. Since $\diam \widetilde{C} > 0$, there exist a subarc $C \subset \widetilde{C}$ such that $r_{0} = d( C, C' \cup ( \partial Y \setminus \widetilde{C} ) ) > 0$. We claim that
\begin{equation}
    \label{eq:intersection:trivial}
    \Mod \Gamma( C, C'; Y \cup C )
    =
    0.
\end{equation}
If \eqref{eq:intersection:trivial} holds for $C$, we obtain a contradiction with \Cref{lemm:positivemodulus}.

So it suffices to prove \eqref{eq:intersection:trivial}. We claim that there exists a sequence $r_{n} \rightarrow 0^{+}$ such that every $\theta \in \Gamma( C, C'; Y \cup C )$ has a subpath in $\Gamma( \phi( E_{r_n} \cap \mathbb{D} ), C'; Y )$. If this can be proved, then the $K$-quasiconformality of $\phi$ yields
\begin{align*}
    \Mod \Gamma( C, C'; Y \cup C )
    &\leq
    K \Mod \Gamma( E_{r_n} \cap \mathbb{D}, \phi^{-1}( C' ); \mathbb{D} ).
\end{align*}
Given \Cref{lemm:points}, the right-hand side converges to zero as $n \rightarrow \infty$, and we conclude \eqref{eq:intersection:trivial}. The rest of the proof is spent on finding the sequence of radii $( r_{n} )_{ n = 1 }^{ \infty }$.

The $K$-quasiconformality of $\phi$ yields that the minimal weak upper gradient $\rho_{\phi}$ satisfies $\rho_{\phi}^{2} \leq K J_{\phi} \in L^{1}( \mathbb{D} )$. The integrability of $J_{\phi}$ follows from the fact that $Y$ has finite Hausdorff $2$-measure. This implies that $\phi$ has an $L^{2}( \mathbb{D} )$-integrable upper gradient $g$ \cite[Lemma 6.2.2]{HKST:15}.

We consider $g_{0} = g \chi_{ \mathbb{D} } \in L^{2}( \overline{\mathbb{D}} )$. Polar coordinates centered at $x_{0}$ yield
\begin{equation}
    \label{eq:integration}
    \infty
    >
    \norm{ g_{0} }_{ L^{1}( \overline{\mathbb{D}} ) }
    \geq
    \int_{ 0 }^{ 2^{-1} }
    \int_{ E_{r} }
        g_{0}
    \,d\mathcal{H}^{1}
    \,d\mathcal{L}^{1}(r).
\end{equation}
In particular, $g_{0}$ has a finite path integral over $E_{r}$ for almost every $0 < r < 2^{-1}$. Let $I$ denote those $0 < r < 2^{-1}$ for which this holds.

Let $\Gamma_{0}$ be the family of non-constant rectifiable paths in $\overline{ \mathbb{D} }$ along which $g_{0}$ is not integrable. Consider an absolutely continuous non-constant path $\gamma \colon \left[a, b\right] \rightarrow \overline{ Y }$ with image in $Y$ and which is not an element of $\Gamma_{0}$. Since $g$ is an upper gradient of $\phi$, we have that
\begin{equation}
    \label{eq:pwcontrol:speed}
    \ell( \phi \circ \gamma )
    \leq
    \int_{ \gamma } g \,ds
    =
    \int_{ \gamma } g_{0} \,ds;
    \text{ see \cite[Proposition 6.3.2]{HKST:15}.}
\end{equation}
Consider a surjective Lipschitz $\gamma_{r} \colon \left[0, 1\right] \rightarrow E_{r}$ for $r \in I$. Then, for every $0 < s < t < 1$, \eqref{eq:pwcontrol:speed} implies
\begin{equation*}
    \ell( ( \phi \circ \gamma_{r} )|_{ \left[s, t\right] } )
    \leq
    \int_{ \gamma_{r}|_{ \left[s, t\right] } } g_{0} \,ds
    \leq
    \int_{ E_{r} } g_{0} \,d\mathcal{H}^{1},
\end{equation*}
where the last term on the right is finite. Therefore $\theta_r = \phi \circ \gamma_r \colon ( 0, 1 ) \rightarrow \phi( E_{r} \cap \mathbb{D} )$ satisfies the assumptions of \Cref{lemm:ACpath}. Hence there exists a continuous extension $\overline{ \theta_{r} } \colon \left[0, 1\right] \rightarrow F_{r}$ onto $F_{r} \coloneqq \overline{ \phi( E_{r} \cap \mathbb{D} ) }$.

Since $\phi$ is a homeomorphism, $\overline{ \theta_{r} }( s ) \not\in Y$ for both $s= 0, 1$. Hence $F_{r}$ is homeomorphic to a circle or an arc. We note that $\phi( V_{r} )$ is one of the connected components of $Y \setminus F_{r}$. In particular, $U_{r} = \overline{ \phi( V_{r} ) }$ is homeomorphic to $\left[0, 1\right]^{2}$, $U_{r} \cap \partial Y$ is a point or an arc, and $C \subset \widetilde{C} \subset U_{r} \cap \partial Y$. As the ends of $U_{r} \cap \partial Y$ and $F_{r}$ coincide, $d( C, \partial Y \setminus \widetilde{C} ) > 0$ implies that $U_{r} \cap \partial Y$ is an arc and $C \cap F_{r} = \emptyset$. This means that every path $\theta \in \Gamma( C, C'; Y \cup C )$ has a subpath in $\Gamma( F_{r} \cap Y, C'; Y )$. Then \eqref{eq:intersection:trivial} follows by taking a sequence $( r_{n} )_{ n = 1 }^{ \infty } \subset I$ converging to zero.
\end{proof}

\begin{lemm}\label{lemm:continuous}
The mapping $\Phi \colon \overline{ \mathbb{D} } \rightarrow \overline{Y}$ is continuous and surjective.
\end{lemm}

\begin{proof}
Let $x_{0} \in \mathbb{S}^{1}$. If $\mathbb{D} \ni x_{n} \rightarrow x_{0}$, the accumulation points of $( \Phi( x_{n} ) )_{ n = 1 }^{ \infty }$ are contained in the intersection of $U_{r}$, where $U_{r}$ are as in the definition of $\widetilde{C}$ in \eqref{eq:intersection:end}. \Cref{lemm:intersection:point} shows that the intersection is a singleton, which we defined to be $\Phi( x_{0} )$. This implies $\Phi( x_{n} ) \rightarrow \Phi( x_{0} )$.
 
More generally, if $\overline{ \mathbb{D} } \ni x_{n} \rightarrow x_{0}$, we find for every $x_{n}$ an element $z_n \in \mathbb{D}$ such that $d_{\overline{Y}}( \Phi( z_n ), \Phi( x_{n} ) ) \leq 2^{-n}$ and $\norm{ z_n - x_{n} } \leq 2^{-n}$. Then $\mathbb{D}\ni z_{n} \rightarrow x_{0}$ and $\Phi( z_{n} ) \rightarrow \Phi( x_{0} )$. Since $d_{Y}( \Phi( z_{n} ), \Phi( x_{n} ) ) \rightarrow 0$, we have $\Phi( x_{n} ) \rightarrow \Phi( x_{0} )$. This implies that $\Phi$ is continuous.

Consider now $y_{0} \in \overline{Y}$. Then there exists a sequence $Y \ni y_{n} \rightarrow y_{0}$. Up to passing to a subsequence and relabeling, $( \Phi^{-1}( y_{n} ) )_{ n = 1 }^{ \infty }$ converges to some $x_{0} \in \overline{\mathbb{D}}$. The continuity of $\Phi$ implies $\Phi( x_{0} ) = y_{0}$.
\end{proof}

\begin{lemm}\label{lemm:surjectivemonotone}
The mapping $\Phi \colon \overline{ \mathbb{D} } \rightarrow \overline{Y}$ is monotone.
\end{lemm}

\begin{proof}
Let $y \in \partial Y$ and suppose that there are two distinct points $x_{1}$ and $x_{2}$ from $\mathbb{S}^{1}$ such that $\Phi( x_{1} ) = y = \Phi( x_{2} )$. Let $I_{i}$ be radial lines from $x_{i}$ to $0$ for $i = 1, 2$ and define $I = I_{0} \cup I_{1}$. Here $\Phi( I )$ is a Jordan loop intersecting $\partial Y$ exactly at $y$. Let $U$ denote the component of $Y \setminus \Phi( I )$ whose closure intersects $\partial Y$ only at $y$. Then $V = \Phi^{-1}( U )$ is one of the components of $\mathbb{D} \setminus I$. By construction, $\overline{V} \cap \mathbb{S}^{1}$ is connected and is mapped to the singleton $y$. Therefore $x_{1}$ and $x_{2}$ can be joined with a path in $\Phi^{-1}( y )$. The monotonicity of $\Phi$ follows.
\end{proof}

\subsection{Proof of \Cref{prop:QS}}
Recall that we fix a quasiconformal homeomorphism $\phi \colon \mathbb{D} \rightarrow Y$ and study its extension $\Phi$ under the assumption \eqref{eq:pointshavezeromod}. This is done in several parts. First, \Cref{lemm:injective} proves that the extension is a homeomorphism. \Cref{lemm:extensionisQC} implies that $\Phi$ is quasiconformal. After this is verified, \Cref{lemm:extensionisQC:samedilatation} proves that $K_{O}( \Phi ) = K_{O}( \phi )$ and $K_{I}( \Phi ) = K_{I}( \phi )$. In particular, the maximal dilatations of $\Phi$ and $\phi$ coincide.

To finish up the proof of \Cref{prop:QS}, we also need to verify that if $\Phi$ is a quasiconformal homeomorphism, then the assumption \eqref{eq:pointshavezeromod} holds. But this follows from the corresponding Euclidean result. So \Cref{prop:QS} follows after we prove the lemmas mentioned above.

\begin{lemm}\label{lemm:injective}
The mapping $\Phi$ is a homeomorphism.
\end{lemm}
\begin{proof}
Let $y \in \partial Y$ and suppose that there exists a non-trivial continuum $E \subset \Phi^{-1}( y )$. Let $F \subset Y \setminus \overline{B}_{ \overline{Y} }( y, R )$ be a non-trivial arc for some $R > 0$. For $0 < r < R$, let $\Gamma( y, r, R ) = \Gamma( \overline{B}_{ \overline{Y} }( y, r ), \overline{Y} \setminus B_{ \overline{Y} }( y, R ); \overline{Y} )$. Then
\begin{equation*}
    \Mod \Gamma( y, r, R )
    \geq
    \Mod( \Gamma( y, r, R ) \cap \mathrm{AC}( Y ) ),
\end{equation*}
where $\mathrm{AC}( Y )$ refers to those absolutely continuous paths in $\overline{Y}$ whose images lie in $Y$. Since $\phi = \Phi|_{ \mathbb{D} }$ is $K$-quasiconformal, we have that
\begin{equation*}
    \Mod( \Gamma( y, r, R ) \cap \mathrm{AC}( Y ) )
    \geq
    K^{-1} \Mod( \Gamma( E, \Phi^{-1}( F ); \mathbb{D} \cup E \cup \Phi^{-1}( F ) ) ).
\end{equation*}
The right-hand side is a strictly positive lower bound; recall \Cref{lemm:positivemodulus}. Therefore $\Mod \Gamma( y, r, R ) \geq C_{0} > 0$ for a constant independent of $r$. By passing to the limit $r \rightarrow 0^{+}$, we find a contradiction with \eqref{eq:pointshavezeromod}. The injectivity of $\Phi$ follows. Since $\Phi$ is continuous, surjective, and monotone, we conclude that $\Phi$ is a homeomorphism.
\end{proof}

Next we claim that $\Phi$ is quasiconformal. To this end, we let $\Psi = \Phi^{-1}$. Due to \Cref{prop:pointwisemodulus}, it is sufficient to find a constant $C_{0}$ such that for every $x \in \mathbb{S}^{1}$,
\begin{equation}
    \label{eq:theinequality:extensionisQC}
    \liminf_{ r \rightarrow 0^{+} }
    \Mod 
    \Psi^{-1} \Gamma
    \left( 
        \overline{B}_{ \overline{\mathbb{D}} }( x,  r ),
        \overline{\mathbb{D}} \setminus B_{ \overline{\mathbb{D}} }( x,  2r );
        \overline{B}_{ \overline{\mathbb{D}} }( x,  2r )
    \right)
    \leq
    C_{0}.
\end{equation}
We denote $\Gamma( x, r, 2r ) = \Gamma \left( \overline{B}_{ \overline{\mathbb{D}} }( x,  r ), \overline{\mathbb{D}} \setminus B_{ \overline{\mathbb{D}} }( x,  2r ); \overline{B}_{ \overline{\mathbb{D}} }( x,  2r ) \right)$ for the rest of the section.

Fix $0 < r < 1/4$. Let $\xi_{1} = \Psi^{-1}\left( \mathbb{S}^{1}( x, r ) \cap \overline{ \mathbb{D} } \right)$ and $\xi_{3} = \Psi^{-1}\left( \mathbb{S}^{1}( x, 2 r  ) \cap \overline{ \mathbb{D} } \right)$. Let $\xi_{2}$ and $\xi_{4}$ denote the subarcs of $\partial Y$ joining $\xi_{1}$ and $\xi_{3}$ in such a way that the arcs $ \xi_{1}, \xi_{2}, \xi_{3}, \xi_{4}$ form the boundary decomposition of a quadrilateral $Q$ in $\overline{Y}$. Then
\begin{equation*}
    \Mod \Psi^{-1} \Gamma( x, r, 2r )
    =
    \Mod \Gamma( \xi_{1}, \xi_{3}; Q )
    \eqqcolon
    M.
\end{equation*}

\begin{lemm}\label{lemm:Kai:Riemann}
There exists a homeomorphism $f \colon Q \rightarrow \left[0, 1\right] \times \left[0, M\right]$ with the following properties: First, $f( \xi_{1} ) = \left\{0\right\} \times \left[0, M\right]$ and $f( \xi_{3} ) = \left\{ 1 \right\} \times \left[0, M\right]$. Whenever $0 < a < b < M$ and $I = \left[a, b\right]$, let $Q^{0} = f^{-1}( \left[0,1\right] \times I )$,
\begin{align*}
    \xi_{1}^{0} &= f^{-1}( \left\{0\right\} \times I ),
    \quad
    \xi_{2}^{0} = f^{-1}( \left[0, 1\right] \times \left\{a\right\} ),
    \\
    \xi_{3}^{0} &= f^{-1}( \left\{1\right\} \times I ),
    \quad
    \xi_{4}^{0} = f^{-1}( \left[0, 1\right] \times \left\{b\right\} ). 
\end{align*}
Then $b-a = \Mod \Gamma( \xi_{ 1}^{0}, \xi_{3}^{0}; Q^{0} )$.
\end{lemm}
\begin{proof}
Proposition 9.1 \cite{Raj:17} and \cite[equation (57), Lemma 10.2]{Raj:17} provide us with $f$ having the stated properties. Notice that \cite[Proposition 9.1]{Raj:17} is applicable since \eqref{point:zero:modulus} holds for every $x_{0} \in \overline{Y}$ and the product in \eqref{upper:bound} is always bounded from below by a universal constant $\kappa_{0} >0$ \cite{RR:19}. These facts allow us to apply \cite[equation (57), Lemma 10.2]{Raj:17} as well.
\end{proof}

\begin{lemm}\label{lemm:extensionisQC}
The inequality \eqref{eq:theinequality:extensionisQC} holds for a constant $C_{0} = 2 K C_{1}$, where $C_{1}$ depends only on $\mathbb{D}$ and $K$ is the maximal dilatation of $\phi$.
\end{lemm}
\begin{proof}
We let $b = 3M/4$ and $a = M/4$ in \Cref{lemm:Kai:Riemann}. Since the restriction of $\Psi$ to $Y$ is $K$-quasiconformal,
\begin{equation*}
     \Mod \Gamma( \xi_{1}^{0}, \xi_{3}^{0}; Q^{0} )
    \leq
    K
    \Mod \Gamma( \Psi( \xi_{1}^{0} ), \Psi( \xi_{3}^{0} ); \Psi( Q^{0} ) ).
\end{equation*}
Observe that $\Gamma( \Psi( \xi_{1}^{0} ), \Psi( \xi_{3}^{0} ); \Psi( Q^{0} ) ) \subset \Gamma( x, r, 2r )$. Therefore
\begin{equation*}
    \Mod \Gamma( \Psi( \xi_{1}^{0} ), \Psi( \xi_{3}^{0} ); \Psi( Q^{0} ) )
    \leq
    \Mod \Gamma( x, r, 2r ).
\end{equation*}
Here $\Mod \Gamma( x, r, 2r ) \leq C_{1}$ for a constant depending only on $\overline{ \mathbb{D} }$. Then \Cref{lemm:Kai:Riemann} yields $M \leq 2 K C_{1}$. The claim follows by passing to the limit $r \rightarrow 0^{+}$.
\end{proof}

\begin{lemm}\label{lemm:extensionisQC:samedilatation}
The outer (resp. inner) dilatation of $\Phi$ coincides with the outer (resp. inner) dilatation of $\phi$.
\end{lemm}
\begin{proof}
Lemmas \ref{lemm:injective} and \ref{lemm:extensionisQC} prove that $\Phi$ is a quasiconformal homeomorphism. In $\mathbb{D}$, the minimal weak upper gradients of $\Phi$ and $\phi$ coincide. This is also true for their Jacobians. Therefore they satisfy (ii) in \Cref{prop:williams:L-Wversion} with the same constant. Hence $K_{O}( \Phi ) = K_{O}( \phi )$.

Consider the Borel functions $g = \chi_{ \mathbb{S}^{1} }$ and $\widetilde{g} = ( g \circ \Phi^{-1} ) \rho_{ \Phi^{-1} }$. The property (ii) in \Cref{prop:williams:L-Wversion} implies
\begin{equation*}
    \norm{ \widetilde{g} }_{ L^{2}( \overline{Y} ) }^{2}
    \leq
    K_{O}( \Phi^{-1} )
    \norm{ g }_{ L^{2}( \overline{ \mathbb{D} } ) }^{2}
    =
    0.
\end{equation*}
Hence $\widetilde{g} = 0$ $\mathcal{H}^{2}_{\overline{Y}}$-almost everywhere in $\overline{Y}$. We conclude that $\rho_{ \Phi^{-1} }(y) = 0$ for $\mathcal{H}^{2}_{ \overline{Y} }$-almost every $y \in \partial Y$. Hence $\rho_{ \Phi^{-1} } = \rho_{ \phi^{-1} } \chi_{ Y }$ $\mathcal{H}^{2}_{ \overline{Y} }$-almost everywhere in $\overline{Y}$. We conclude that $\phi$ and $\Phi$ satisfy (ii) in \Cref{prop:williams:L-Wversion} with the same constant. In other words, $K_{I}( \Phi ) = K_{I}( \phi )$.
\end{proof}

\section{Beurling--Ahlfors extension}\label{sec:BA}
For this section we fix a quasiconformal Jordan domain $Y$ satisfying \eqref{eq:area:growth:intro} and a quasiconformal homeomorphism $\phi \colon \mathbb{D} \rightarrow Y$. Let $\Phi \colon \overline{\mathbb{D}} \rightarrow \overline{Y}$ denote the quasiconformal homeomorphic extension of $\phi$, obtained from \Cref{prop:QS}. We refer to $g_{\phi} = \Phi|_{ \mathbb{S}^{1} }$ as the boundary map of $\phi$. The goal of this section is to prove \Cref{thm:QS}. We reduce the proof to \Cref{lemm:QS:boundary}.

Observe that if $g_{\phi}$ is a quasisymmetry, then $\partial Y$ has bounded turning as this property is preserved by quasisymmetries \cite{Tuk:Vai:80}. Moreover, if we fix an arbitrary quasisymmetry $g \colon \mathbb{S}^{1} \rightarrow \partial Y$, then $h = g_{\phi}^{-1} \circ g \colon \mathbb{S}^{1} \rightarrow \mathbb{S}^{1}$ is a quasisymmetry. The Beurling--Ahlfors extension theorem \cite{Ah:Beu:56} yields the existence of a quasiconformal map $H \colon \overline{ \mathbb{D} } \rightarrow \overline{ \mathbb{D} }$ whose boundary map equals $h$. Then $G = \Phi \circ H \colon \overline{ \mathbb{D} } \rightarrow \overline{ Y }$ is the quasiconformal extension of $g$ whose existence we wanted to establish. So \Cref{thm:QS} is a consequence of the following result.

\begin{prop}\label{lemm:QS:boundary}
If $\partial Y$ has bounded turning and satisfies the mass upper bound \eqref{eq:area:growth:intro}, then the boundary map $g_{\phi} = \Phi|_{ \mathbb{S}^{1} }$ is a quasisymmetry.
\end{prop}
We start the proof of \Cref{lemm:QS:boundary} by first establishing \Cref{lemm:assouad}. There we claim that $\partial Y$ has Assouad dimension at most two.

\begin{lemm}\label{lemm:Ahlfors:lowerbound}
Suppose that $\partial Y$ has bounded turning with constant $\lambda > 1$. Let $C = ( 4 \lambda )^{3} \frac{ 2 }{ \pi }$. Then for all $y \in \partial Y$ and all $0 < r < \diam \partial Y$, $\mathcal{H}^{2}_{ \overline{Y} }( \overline{B}_{ \overline{Y} }( y, r ) ) \geq C^{-1} r^{2}.$
\end{lemm}
\begin{proof}
Consider $y \in \partial Y$ and the $1$-Lipschitz function $f(z) = d( y, z )$. For every $0 < r < ( 4 \lambda )^{-1}\diam \partial Y$, we have that $f^{-1}( r ) \cap \partial Y \neq \emptyset$. Let $y_{0} \in \partial Y \setminus \overline{B}_{\overline{Y}}( y, 2 \lambda r )$. We obtain from \Cref{thm:topology} a continuum $C_{r} \subset f^{-1}( r )$ separating $y$ and $y_{0}$.

Let $E \subset \partial Y$ be the (closure of the) component of $\partial Y \setminus C_{r}$ that contains $y$ and let $a$ and $b$ denote the ends of $E$. Here $a, b \in E \cap C_{r} \subset f^{-1}(r)$ so $d(a,y) = r = d(b,y)$.

Let $E_{a} \subset \partial Y$ be the arc ending at $a$ and $y$ with $\diam E_{a} \leq \lambda r$, and let $E_{b}$ denote the corresponding arc for $b$ and $y$. Then $E_{a} \cup E_{b} = E$. Indeed, otherwise $\max\left\{ \diam E_{a}, \diam E_{b} \right\} \geq d( y, y_{0} ) > 2 \lambda r$.

Let $E' \subset \partial Y$ be the arc ending at $a$ and $b$ with smaller diameter. Then $\diam E' \leq 2^{-1} \diam \partial Y$. Since $\diam E \leq \diam E_{a} + \diam E_{b} \leq 2 \lambda r < 2^{-1} \diam \partial Y$, we have $E' = E$.

We conclude that $\diam C_{r} \geq d( a, b ) \geq \lambda^{-1} \diam E \geq \lambda^{-1} r$. Since $C_{r}$ is a continuum, $\mathcal{H}^{1}_{ \overline{Y} }( f^{-1}( r ) ) \geq \mathcal{H}^{1}_{ \overline{Y} }( C_{r} ) \geq \diam C_{r} \geq \lambda^{-1} r$ for each $0 < r < ( 4 \lambda )^{-1} \diam \partial Y$ \cite[2.10.12]{Fed:69}.

By integrating over the interval $\left(0, r\right)$ and by applying \Cref{thm:eilenberg}, we conclude that
\begin{equation*}
    \mathcal{H}^{2}_{ \overline{Y} }( \overline{B}_{ \overline{Y} }( y, r ) )
    \geq
    ( 4 \lambda )^{-1}
    \frac{ \pi }{ 2 }
    r^{2}
\end{equation*}
whenever $0 < r < ( 4 \lambda )^{-1} \diam \partial Y$. If $( 4 \lambda )^{-1} \diam \partial Y \leq r < \diam \partial Y$, we have
\begin{equation*}
    \mathcal{H}^{2}_{ \overline{Y} }( \overline{B}_{ \overline{Y} }( y, r ) )
    \geq
    \mathcal{H}^{2}_{ \overline{Y} }( \overline{B}_{ \overline{Y} }( y, ( 4 \lambda )^{-1} r ) )
    \geq
    ( 4 \lambda )^{-1}
    \frac{ \pi }{ 2 }
    \left( ( 4 \lambda )^{-1} r \right)^{2}.
\end{equation*}
The claim follows for $C = ( 4 \lambda )^{3} \frac{ 2 }{ \pi }$.
\end{proof}

\begin{proof}[Proof of \Cref{lemm:assouad}]
Let $0 < r < \diam \partial Y$ and $0 < \epsilon < 1$. We consider a point $y \in \partial Y$ and an $\epsilon r$-separated set $S \subset B_{ \overline{Y} }( y, r ) \cap \partial Y$. We conclude from \Cref{lemm:Ahlfors:lowerbound} and \eqref{eq:area:growth:intro} that for some $C \geq 1$
\begin{equation*}
    C ( 2r )^{2}
    \geq
    \mathcal{H}^{2}_{ \overline{Y} }( \overline{B}_{ \overline{Y} }( y, 2r ) )
    \geq
    \mathcal{H}^{2}_{ \overline{Y} }\left( \bigcup_{ x \in S } \overline{B}_{ \overline{Y} }( x, \epsilon r ) \right)
    \geq
    C^{-1}\left( \frac{ \epsilon r }{ 2 } \right)^{2}
    { \# S }.
\end{equation*}
Therefore ${ \# S } \leq C \epsilon^{-2}$ for some constant $C$ independent of $r$, $y$ and $\epsilon$. Hence the Assouad dimension of $\partial Y$ is at most $2$.
\end{proof}

\begin{proof}[Proof of \Cref{lemm:QS:boundary}]
Let $0 < r_{0}$ be such that for every $x \in \mathbb{S}^{1}$ and every $0 < r \leq r_{0}$, we have that $4 L_{ g }( x, 2 r ) < \diam \partial Y$.

Fix $x \in \mathbb{S}^{1}$ and $0 < r < r_{0}$. Let $z, a, b \in \mathbb{S}^{1} \cap \mathbb{D}( x, r )$ be such that $0 < \norm{ a - z }_{2} \leq  \norm{ b - z }_{2}$. We proved in \Cref{lemm:assouad} that $\partial Y$ is doubling, so by a result of Tukia--Väisälä \cite[Theorems 2.15 and 2.23]{Tuk:Vai:80}, the quasisymmetry of $g = g_{\phi}$ follows if there exists a constant $H > 0$ depending only on the constants in \eqref{eq:area:growth:intro}, \Cref{lemm:Ahlfors:lowerbound}, the bounded turning constant $\lambda$ of $\partial Y$, and the maximal dilatation $K$ of $\phi$ such that $d( g(a), g(z) ) \leq H d( g(b), g(z) )$.

Let $\ell = d( g(a), g(z) )$ and let $M > 0$ be such that $\ell > M d( g(b), g(z) )$. If we find $H_{0}$ such that $M \leq H_{0}$ independently of $a,b$ and $z$, we may set $H = H_{0}$. If $M \leq ( 2 \lambda )^{2}$, any choice $H_{0} \geq ( 2 \lambda )^{2}$ suffices. So we may assume $M > ( 2 \lambda )^{2}$.

Fix $z' \in \mathbb{S}^{1}$ with $d( g(z'), g(x) ) > 2 L_{g}( x, 2 r )$. Then
\begin{align*}
    2L_{g}( x, 2r )
    &<
    d( g(z'), g(z) )
    +
    d( g(z), g(x) )
    \leq
    d( g(z'), g(z) ) + L_{g}( x, 2 r )
    \quad \text{and}
    \\
    2^{-1} \ell
    &\leq
    2^{-1}( d( g(a), g(x) ) + d( g(x), g(z) ) )
    \leq
    L_{g}( x, 2 r ).
\end{align*}
Therefore $2^{-1} \ell < d( g(z'), g(z) )$. We conclude that
\begin{equation*}
    g(a), g(z')
    \in
    \partial Y \setminus \overline{B}_{ \overline{Y} }\left( g(z), \frac{ \ell }{ 2 } \right)
    \quad\text{and}\quad
    g(b)
    \in
    \partial Y \cap \overline{B}_{ \overline{Y} }\left( g(z), \frac{ \ell }{ M } \right).
\end{equation*}
Let $A'$ be the subarc of $\partial Y$ joining $g(a)$ to $g(z')$ that does not contain $g(z)$. Then any arc joining $A'$ to $g( z )$ within $\partial Y$ must pass through either $g( a )$ or $g( z' )$. Using this fact, the bounded turning of $\partial Y$ yields
\begin{equation}
    \label{eq:Aprime}
    A'
    \subset
    \partial Y \setminus \overline{B}_{ \overline{Y} }\left( g(z), \lambda^{-1} \frac{ \ell }{ 2 } \right).
\end{equation}
Let $B'$ be the subarc of $\partial Y$ with smallest diameter which ends at $g(b)$ and $g(z)$. The bounded turning of $\partial Y$ implies that
\begin{equation}
    \label{eq:Bprime}
    B'
    \subset
    \overline{B}_{ \overline{Y} }\left( g(z), \lambda \frac{ \ell }{ M } \right).
\end{equation}
The inclusions \eqref{eq:Aprime} and \eqref{eq:Bprime} imply that every path $\gamma \in \Gamma( A', B'; \overline{Y} )$ has a subpath joining $\overline{Y} \setminus \overline{B}_{ \overline{Y} }( g(z), \lambda^{-1} \frac{ \ell }{ 2 } )$ to $\overline{B}_{ \overline{Y} }\left( g(z), \lambda \frac{ \ell }{ M } \right)$ within $\overline{B}_{ \overline{Y} }( g(z), \lambda^{-1} \frac{ \ell }{ 2 } )$. Then \Cref{lemm:points} yields
\begin{equation}
    \label{eq:modulus:bound:basic}
    \frac{ \widetilde{C}_{U} }{ \log \frac{ M }{ 2 \lambda^{2} } }
    \geq
    \Mod \Gamma( A', B'; \overline{Y} ).
\end{equation}
Let $A = g^{-1}( A' )$ and $B = g^{-1}( B' )$. The relative distance $\Delta( A, B )$ satisfies
\begin{equation}
    \label{eq:relative:distance:upper}
    \Delta( A, B )
    \coloneqq
    \frac{ d( A, B ) }{ \min\left\{ \diam A, \diam B \right\} }
    \leq
    2.
\end{equation}
First, $d( A, B ) \leq \norm{ a - z }_{2}$ since $a \in A$ and $z \in B$. Second, $\diam A \geq \norm{ a - z' }_{2} \geq r$, so $\norm{ a - x }_{2} \leq r$ and $\norm{ x - z }_{2} \leq r$ imply $2^{-1} \norm{ a - z }_{2} \leq r$. Lastly, $\diam B \geq \norm{ b - z }_{2} \geq \norm{ a - z }_{2}$. These imply \eqref{eq:relative:distance:upper}.

The 2-Loewner property of $\overline{\mathbb{D}}$ \cite[Example 8.24]{Hei:01} states that there exists a constant $C_{2} > 0$ for which
\begin{equation}
    \label{eq:modulus:bound:loewner}
    \Mod \Gamma( A, B; \overline{ \mathbb{D} } )
    \geq
    C_{2}
\end{equation}
depending only on the upper bound \eqref{eq:relative:distance:upper}. The $K$-quasiconformality of the extension $\Phi$ implies that $\Mod \Gamma( A', B'; \overline{ Y } ) \geq K^{-1} \Mod \Gamma( A, B; \overline{ \mathbb{D} } )$. Combining this inequality with \eqref{eq:modulus:bound:basic} and \eqref{eq:modulus:bound:loewner} yields an upper bound on $M$ in terms of $C_{2}$, $K$, $\widetilde{C}_{U}$, and $\lambda$. Setting $H_{0}$ to be the maximum of this bound and $( 2 \lambda )^2$ establishes the claim.
\end{proof}

\section{Planar quasicircles}\label{sec:PQ}

We prove \Cref{prop:planar:QC} in this section. The main result of this section is the following.
\begin{prop}\label{eq:boundary:Assouad:planar}
Under the assumptions of \Cref{prop:planar:QC}, the Assouad dimension of $\partial Y$ is strictly less than $2$.
\end{prop}

\begin{proof}[Proof of \Cref{prop:planar:QC} assuming \Cref{eq:boundary:Assouad:planar}.]
\Cref{eq:boundary:Assouad:planar} states that $\partial Y$ has Assouad dimension strictly less than $2$. Having verified this, \cite{Her:Mey:12} yields the existence of a bi-Lipschitz embedding $h \colon \partial Y \rightarrow \mathbb{R}^{2}$, i.e., $\partial Y$ is planar.
\end{proof}

So \Cref{prop:planar:QC} follows from \Cref{eq:boundary:Assouad:planar}. We split the proof of \Cref{eq:boundary:Assouad:planar} into a couple of sublemmas. 

\begin{lemm}\label{eq:boundary:porosity}
Suppose that $\partial Y$ has $\lambda$-bounded turning and satisfies \ref{prop:planar:QC.1} in \Cref{prop:planar:QC}. Then there exists a constant $C_{p} \geq 1$ depending only on $\lambda$ such that for every $x \in \partial Y$ and $0 < r < \min\left\{ r_{0}, \diam \partial Y \right\}$, there exists $y \in Y$ with $\overline{B}_{ \overline{Y} }( y, C_{p}^{-1} 2 r ) \subset B_{ \overline{Y} }( x, r ) \setminus \partial Y$.
\end{lemm}

\begin{proof}
This is only a small modification of the proof of Theorem 8.2 of \cite{Mer:Wil:13} but we include the details here for the convenience of the reader.  Let $s = 8 \lambda^{2}( 2 \lambda + 1 )$ and $C_{p} = 8 \lambda s$. Let $x \in \partial Y$ and $0 < r < \min\left\{ r_{0}, \diam \partial Y \right\}$. We claim that there exists a point $v \in Y$ such that
\begin{equation}
    \label{eq:boundary:porosity:proof}
    B_{ \overline{Y} }\left( v, C_{p}^{-1} 2 r \right)
    \subset
    B_{ \overline{Y} }( x, r )
    \setminus
    \partial Y. 
\end{equation}

Suppose for now that $r < ( 4 \lambda )^{-1} \min\left\{ r_{0}, \diam \partial Y \right\}$. Then there exists a point $z \in \partial Y$ for which $d( x, z ) \geq 2 \lambda r$. We fix such a $z$.

Let $|\gamma|$ denote the (closure of the) subarc of $\partial Y \setminus \left\{ w \in \partial Y \colon d( x, w ) = \frac{ r }{ 4 \lambda } \right\}$ that contains $x$. Let $a$ and $b$ denote the end points of $|\gamma|$ labelled in such a way that $\left\{ x, a, z, b \right\}$ is cyclically ordered on $\partial Y$. We have that $d( x, a ) = d( x,  b ) = \frac{ r }{ 4 \lambda }$ and $|\gamma| \subset \overline{B}_{ \overline{Y} }\left( x, \frac{ r }{ 4 \lambda } \right)$.

The relative ALLC condition of $\partial Y$ implies that there exists a path $\alpha$ joining $a$ to $b$ in $\overline{B}_{ \overline{Y} }\left( x, \frac{ r }{ 4 } \right) \setminus B_{ \overline{Y} }\left( x, \frac{ r }{ 8 \lambda^{2} } \right)$. We assume without loss of generality that $\alpha$ is an arc.

Let $|\gamma_{a}|$ denote the (closure of the) component of $\partial Y \setminus \left\{ x, z \right\}$ joining $x$ and $z$ that contains $a$ and let $|\gamma_{b}|$ be the other component. Observe that $d( a, |\gamma_{b}| ) \geq ( 8 \lambda^{2} )^{-1} r$ since otherwise we would find an arc $|\gamma'|$ joining $|\gamma_{b}|$ to $a$ within $\partial Y$ for which $( 8 \lambda )^{-1} r \geq \diam |\gamma'|$. This would imply the contradiction $( 8 \lambda )^{-1} r \geq d( a, \left\{ x, z \right\} )$.

The lower bound on $d( a, |\gamma_{b}| )$ and connectedness of $| \alpha |$ imply the existence of $v \in |\alpha|$ such that $d( v, |\gamma_{b}| ) = \frac{ r }{ s }$. Fix such a $v$. Suppose that there exists $w \in |\gamma_{a}|$ for which $d( v, w ) < \frac{ r }{ s }$. Then $d( w, |\gamma_{b}| ) < 2 \frac{ r }{ s }$ and there exists a path $\beta'$ joining $w$ to $| \gamma_{b} |$ within $\partial Y$ for which $\diam |\beta'| < 2 \lambda \frac{ r }{ s }$. Since $|\beta'|$ contains either $x$ or $z$, we have
\begin{equation*}
    d( v, \left\{ x, z \right\} )
    \leq
    d( w, \left\{ x, z \right\} )
    +
    d( v, w )
    <
    2 \lambda
    \frac{ r }{ s }
    +
    \frac{ r }{ s }
    =
    \frac{ r }{ 8 \lambda^{2} }.
\end{equation*}
This is a contradiction with the facts $d( x, z ) > 2\lambda r$ and $v \in |\alpha| \subset \overline{B}_{ \overline{Y} }\left( x, \frac{ r }{ 4 } \right) \setminus B_{ \overline{Y} }\left( x, \frac{ r }{ 8 \lambda^{2} } \right)$. Since no such $w$ exists, $d( v, |\gamma_{a}| ) \geq \frac{ r }{ s }$. Consequently, $B_{ \overline{Y} }\left( v, \frac{ r }{ s } \right) \subset B_{ \overline{Y} }( x, r ) \setminus \partial Y$.

If $( 4 \lambda )^{-1} \min\left\{ r_{0}, \diam \partial Y \right\} \leq r < \min\left\{ r_{0}, \diam \partial Y \right\}$, then there exists a point $v \in Y$ such that
\begin{equation*}
    B_{ \overline{Y} }\left( v, \frac{ r }{ 4 \lambda s } \right)
    \subset
    B_{ \overline{Y} }\left( x, \frac{ r }{ 4 \lambda } \right) \setminus \partial Y
    \subset
    B_{ \overline{Y} }( x, r )
    \setminus
    \partial Y. 
\end{equation*}
In either case, \eqref{eq:boundary:porosity:proof} holds.
\end{proof}

Let $r_{1} = \min\left\{ r_{0}, \diam \partial Y \right\}$, where $r_{0}$ is the parameter from the assumptions of \Cref{prop:planar:QC}. There exists a constant $C \geq 1$ with the following properties:
\begin{enumerate}[label=(\roman*)]
    \item\label{proof:eq:1} For every $y \in \partial Y$ and every $0 < r$, $\mathcal{H}^{2}_{ \overline{Y} }( B_{ \overline{Y} }( y, r ) ) \leq C r^{2}$. Moreover, for every $y \in \partial Y$ and every $0 < r < r_1$, $\mathcal{H}^{2}_{ \overline{Y} }( B_{ \overline{Y}}( y, r ) ) \geq C^{-1} r^{2}$.
    \item\label{proof:eq:2} For every $y \in \partial Y$ and $0 < r < r_1$, there exists $z \in Y$ such that $\overline{B}_{ \overline{Y} }( z, C^{-1} 2 r ) \subset B_{ \overline{Y} }( y, r ) \setminus \partial Y$.
    \item\label{proof:eq:3} For every $z \in Y$ with $d( z, \partial Y ) \leq r_1$ and every $0 < r < d( z, \partial Y )$, $C^{-1} r^{2} \leq \mathcal{H}^{2}_{ \overline{Y} }( B_{ \overline{Y} }( z, r ) )$.
\end{enumerate}
\begin{rem}
We have assumed that for all radii $0 < r < \diam \partial Y$ and every $y \in \partial Y$, $\mathcal{H}^{2}_{ \overline{Y} }( B_{ \overline{Y} }( y, r ) ) \leq C' r^{2}$ for some $C' > 0 $. Then the upper bound in \ref{proof:eq:1} holds for every $r > 0$ if we replace $C'$ with $C'' = \max\left\{ C', \mathcal{H}^{2}_{ \overline{Y} }( \overline{Y} ) / ( \diam \partial Y )^{2} \right\}$. The lower bound for such balls follows from \Cref{lemm:Ahlfors:lowerbound}.

The property \ref{proof:eq:2} follows from \Cref{eq:boundary:porosity} for some constant. Recall that, under the assumptions of \Cref{prop:planar:QC}, the lower bound in \ref{proof:eq:3} holds for each $0 < r < d( z, \partial Y ) \leq r_{0}$ for some constant $C'$. Hence \ref{proof:eq:3} follows.

Our claim is qualitative, so we use a uniform constant $C$ for these various conditions.
\end{rem}
In the following, if $B$ is a ball and $\xi > 0$, $\xi B$ refers to the ball with the same center and $\xi$ times the radius.

\begin{lemm}\label{lemm:5r-covering}
There exists a collection $\mathcal{B}$ of pairwise disjoint balls in $Y = \overline{Y} \setminus \partial Y$ such that for every $( y, r ) \in \partial Y  \times ( 0, r_1 )$ there exists a ball $B \in \mathcal{B}$ with
\begin{equation}
    \label{eq:comparable:constants}
    r_{1}/2
    >
    \max\left\{ d( \partial Y, B ), \diam B \right\}
    \quad\text{and}\quad
    \diam B \simeq r \simeq d( y, B ) \simeq d( \partial Y, B ),
\end{equation}
where $A_{1} \simeq A_{2}$ means that there exists a \emph{constant of comparability} $D > 0$ for which $D^{-1} A_{1} \leq A_{2} \leq D A_{1}$. Here the constants of comparability depend only on the constant $C$.
\end{lemm}
\begin{proof}[Proof of \Cref{lemm:5r-covering}]
For each $x \in \partial Y$ and each integer $m$ such that $0 < 2^{m} < r_1$, consider a point $v_{x,m} \in Y = \overline{Y} \setminus \partial Y$ with $\overline{B}_{ \overline{Y} }( v_{x,m}, C^{-1} 2^{m+1} ) \subset B_{ \overline{Y} }( x, 2^{m} ) \setminus \partial Y$. Then the ball $B_{x,m} \coloneqq \overline{B}_{ \overline{Y} }( v_{x,m}, C^{-1} 2^{m} )$ satisfies
\begin{align}
    \label{eq:comparability}
    \diam B_{x,m}
    \simeq
    2^{m}
    \simeq
    d( x, B_{x,m} )
    \simeq
    d( \partial Y, B_{x, m} )
\end{align}
with the constants of comparison depending only on $C$.

For each $m \in \mathbb{Z}$ with $2^{m} < r_1$, let $\mathcal{B}_{m}$ denote the collection of the $B_{x,m}$ as $x \in \partial Y$ varies. The $5r$-covering theorem \cite[2.8.4]{Fed:69} states that there exists subcollection $\mathcal{B}_{m}' \subset \mathcal{B}_{m}$ whose elements are pairwise disjoint and
\begin{equation}
    \label{eq:5r-covering}
    \bigcup_{ B \in \mathcal{B}_{m} }
        B
    \subset
    \bigcup_{ B \in \mathcal{B}_{m}' }
        5 B.
\end{equation}
Let $1 \leq N \in \mathbb{N}$ and $m_{1} \in \mathbb{Z}$ with $2^{ m_{1} } < r_{1} \leq 2^{ m_{1} + 1 }$. Consider the collection $\mathcal{B} = \bigcup_{ k = 1 }^{ \infty } \mathcal{B}_{ m_{1} - k N }'$.

By choosing a sufficiently large $N$, the elements of $\mathcal{B}$ are pairwise disjoint and each element satisfies $\max\left\{ d( \partial Y, B ), \diam B \right\} < r_1/2$. The choice of $N$ depends only on the constants in \eqref{eq:comparability}. The inclusion \eqref{eq:5r-covering} implies that for each $0 < r < r_1$ and $x \in \partial Y$, there exists $B \in \mathcal{B}$ such that \eqref{eq:comparability} holds with $B_{x,m}$ replaced by $B$ and $2^{m}$ by $r$, with constants of comparison depending only on $C$. Then \eqref{eq:comparable:constants} follows.
\end{proof}

We fix arbitrary $0 < \epsilon < 1$, $0 < r < r_1$ and $y \in \partial Y$. Let $N = B_{ \overline{Y} }( y, 2 r ) \cap \left\{ x \in \overline{Y} \colon d( \partial Y, x ) < 2^{-1} \epsilon r \right\}$. Let $\mathcal{B}_{0} \subset \mathcal{B}$ denote the subcollection consisting of all $B \in \mathcal{B}$ with $B \subset B_{ \overline{Y} }( y, 4 r )$ with diameter at least $a \epsilon r$ for $a > 0$.

\begin{lemm}\label{lemm:accumulation:balls}
There exist $1 > a >  0$, $b > 0$ and $\xi > 1$, with the choices depending only on the constants of comparability in \eqref{eq:comparable:constants}, such that
\begin{equation}
    \label{eq:countingfunction}
    T( x )
    \coloneqq
    \sum_{ B \in \mathcal{B}_{0} }
    \chi_{ \xi B }(x)
    \geq
    -b \log( \epsilon )
    \quad \text{ for all } x \in N,
\end{equation}
where $\chi_{ \xi B }$ is the characteristic function of the ball $\xi B$.
\end{lemm}

\begin{proof}[Proof of \Cref{lemm:accumulation:balls}]
For each $( w, s ) \in B( y, 3 r ) \cap \partial Y \times ( 0, r_1 )$, let $B_{s}( w ) \in \mathcal{B}$ be the ball obtained from \Cref{lemm:5r-covering}. \Cref{lemm:5r-covering} implies the existence of $A > 1$ for which $A^{-1} s \leq \diam B_{s}(w) < A s$ and for the center $c_{s}(w)$ of $B_{s}(w)$, $A^{-1} s \leq d( w, c_{s}(w) ) < A s$.

Let $0 < a < A^{-2}$. Then for every $A a \epsilon r \leq s < r/A$, we have $B_{s}( w ) \in \mathcal{B}_{0}$. Moreover, if $\xi \geq 3 a^{-1}$, the radius of $\xi B_{s}(w)$ is bounded from below by $3 a^{-1} \diam B_{s}(w) / 2$. Given $z \in B( w, 2^{-1} \epsilon r $), $$d( c_{s}(w), z ) < s/( 2 A a ) + d( c_{s}(w), w ) \leq ( a^{-1} + 2A^{2} ) \frac{ \diam B_{s}( w ) }{ 2 }.$$ Thus $\xi B_{s}(w) \supset B( w, 2^{-1} \epsilon r )$.

Let $k \in \mathbb{Z}$ be the largest integer for which $A a \epsilon r < r/A^{2k+1}$. Let $\left\{ \widetilde{s}_{i} \right\}_{ i = 1 }^{ k}$ be a strictly increasing sequence in the interval $\left( A a \epsilon r, r / A^{2k+1} \right)$. Denote $s_{i} = A^{ 2(i-1) } \widetilde{s}_{i}$ for each $i = 1,2,\dots k$. Here $s_{k} < r/A$ and $A s_{i-1} < A^{-1} s_{i}$ for each $i$. Hence the collection $\left\{ B_{s_{i}}( w ) \right\}_{ i = 1 }^{ k}$ contains $k$ different balls. This implies
\begin{equation}
    \label{eq:countingfunction:w}
    T(z)
    \geq
    k
    \quad
    \text{for every $z \in B( w, 2^{-1} \epsilon r )$}.
\end{equation}
We set now $a = A^{-4}$, $\xi = 3 a^{-1}$, and $b = 1/( 2 \log( A ) )$. The maximality of $k$ implies $k \geq - b \log( \epsilon )$.

If $z \in N$, there exists $w_{z} \in \partial Y$ such that $d( w_{z}, z ) = d( \partial Y, z ) < 2^{-1} \epsilon r$. In particular, $w_{z} \in B( y, 3r ) \cap \partial Y$ and $z \in B( w_{z}, 2^{-1}\epsilon r )$. Then \eqref{eq:countingfunction:w} implies \eqref{eq:countingfunction} for the constants $a$, $\xi$, and $b$.
\end{proof}

\begin{lemm}\label{lemm:doublingcondition}
Let $a, b, \xi$, and $\mathcal{B}_{0}$ be as in \Cref{lemm:accumulation:balls}. There exists a constant $d \geq 1$, depending only on the constants of comparability in \eqref{eq:comparable:constants} and $C$, such that $\mathcal{H}^{2}_{ \overline{Y} }( 5 \xi B ) \leq d \mathcal{H}^{2}_{ \overline{Y} }( B )$ for every $B \in \mathcal{B}_{0}$.
\end{lemm}

\begin{proof}[Proof of \Cref{lemm:doublingcondition}]
Consider $B \in \mathcal{B}_{0}$. Then $5 \xi B \subset B_{ \overline{Y} }( y, \rho )$ for some $y \in \partial Y$ such that $5 \xi \diam B \simeq 5 \xi d( y, B ) \simeq 5 \xi d( \partial Y, B ) \simeq \rho$, depending only on the constants of comparability in \eqref{eq:comparable:constants}. The mass upper bound \ref{proof:eq:1} yields $\mathcal{H}^{2}_{ \overline{Y} }( 5 \xi B ) \leq C \rho^{2} \simeq C 25 \xi^{2} (\diam B)^{2}$. Given \eqref{eq:comparable:constants}, we have $\max\left\{ \diam B, d( B, \partial Y ) \right\} < r_{1}/2$. Hence the lower bound from \ref{proof:eq:3} implies $( \diam B )^{2} \leq C \mathcal{H}^{2}_{ \overline{Y} }( B )$. The existence of $d$ follows.
\end{proof}

\begin{proof}[Proof of \Cref{eq:boundary:Assouad:planar}]
The claim is that there exist $0 < \delta < 2$ and $\widetilde{C} > 0$ such that for every $0 < \epsilon < 1$, every $0 < r < \diam \partial Y$, and every $y \in \partial Y$, any $\epsilon r$-separated set $E \subset B( y, r ) \cap \partial Y$ satisfies ${\#}E \leq \widetilde{C} \epsilon^{ -\delta }$.

Suppose that the claim holds whenever $0 < r < r_1$. Consider $r_1 \leq r < \diam \partial Y$. Let $E_{0} \subset \partial Y$ be a maximal $r_1 / 2$-separated net. For every $f \in E_{0}$, the set $E_{f} = B( f, r_1 / 2 ) \cap E$ is $\epsilon r_1 / 2$-separated. Since $E = \bigcup_{ f \in E_{0} } E_{f}$, we have ${\#} E \leq \sum_{ f \in E_{0} } {\#} E_{f} \leq {\#}E_{0} \widetilde{C} \epsilon^{-\delta}$. So the general case follows from the special one.

Now we prove the claim for each $0 < r < r_1$, $0 < \epsilon < 1$, and $y \in \partial Y$. We choose $a$, $b$ and $\xi$ as in \Cref{lemm:accumulation:balls} and let $\mathcal{B}_{0} \subset \mathcal{B}$ be the collection defined before the statement \Cref{lemm:accumulation:balls}. The collection $\mathcal{B}_{0}$ has the following properties:
\begin{enumerate}
    \item[$A_1$.] $\sup_{ B \in \mathcal{B}_{0} } \diam B < \infty$;
    \item[$A_{2}$.] $0 < \mathcal{H}^{2}_{ \overline{Y} }( 5 \xi B ) \leq d \mathcal{H}^{2}_{ \overline{Y} }( B )$ for each $B \in \mathcal{B}_{0}$;
    \item[$A_{3}$.] the balls in $\mathcal{B}_{0}$ are pairwise disjoint;
    \item[$A_{4}$.] the measure of $S \coloneqq \bigcup_{ B \in \mathcal{B}_{0} } B \subset B_{ \overline{Y} }( y, 4 r )$ is finite.
\end{enumerate}
\Cref{lemm:doublingcondition} implies that the constant $d$ in $A_{2}$ can be chosen to be independent of $y$, $r$, and $\epsilon$.

Having verified properties $A_{1}$-$A_{4}$, \cite[Theorem 9.6 (b)]{Bo:Kos:Roh:98} explicitly states for $\mu = 1/12d^{2}$ the following:
\begin{equation*}
    \mathcal{H}^{2}_{ \overline{Y} }\left(
        \left\{ x \in \overline{B}_{ \overline{Y} }( y, 2 r ) \colon T( x ) \geq -b \log( \epsilon ) \right\}
    \right)
    \leq
    ( 1 + d ) \mathcal{H}^{2}_{ \overline{Y} }( S ) e^{ -\mu ( - b \log( \epsilon ) ) }.
\end{equation*}
Given that $S \subset B_{ \overline{Y} }( y, 4 r )$, the upper bound in \ref{proof:eq:1} implies
\begin{equation}
    \label{eq:upperbound:mass}
    \mathcal{H}^{2}_{ \overline{Y} }\left(
        \left\{ x \in \overline{B}_{ \overline{Y} }( y, 2 r ) \colon T( x ) \geq -b \log( \epsilon ) \right\}
    \right)
    \leq
    ( 1 + d )( C (4r)^{2} ) \epsilon^{ \mu b }.
\end{equation}
Let $E \subset B_{ \overline{Y} }( y, r ) \cap \partial Y$ be an $\epsilon r$-separated set. We see from \eqref{eq:countingfunction} and the lower bound in \ref{proof:eq:1} that
\begin{align}
    \label{eq:lowerbound:mass}
    &\mathcal{H}^{2}_{ \overline{Y} }\left(
        \left\{ x \in \overline{B}_{ \overline{Y} }( y, 2 r ) \colon T( x ) \geq -b \log( \epsilon ) \right\}
    \right)
    \\ \notag
    \geq \,
    &\mathcal{H}^{2}_{ \overline{Y} }\left( N \right)
    \geq
    \mathcal{H}^{2}_{ \overline{Y} }\left( \bigcup_{ z \in E } B( z, 2^{-1}\epsilon r ) \right)
    \geq
    ( {\# E} ) C^{-1} (2^{-1} \epsilon r )^{2}.
\end{align}
Now \eqref{eq:upperbound:mass} and \eqref{eq:lowerbound:mass} yield ${\# E} \leq \widetilde{C} \epsilon^{ -( 2 - \mu b ) }$, where $\widetilde{C}$ is independent of $y$, $r$, and $\epsilon$. We denote $\delta \coloneqq 2 - \mu b < 2$. The claim follows.
\end{proof}

\section{Concluding remarks}\label{sec:conc}
Consider any quasiconformal Jordan domain $Y$ and $\phi$ as in \Cref{thm:carat:metric}. Since $\phi$ is a homeomorphism, the Jacobian $J_{\phi}$ of $\phi$ satisfies
\begin{equation}
    \label{eq:boundedarea}
    \mathrm{Area}( \phi )
    \coloneqq
    \int_{ \mathbb{D} } J_{\phi} \,d\mathcal{L}^{2}
    \leq
    \mathcal{H}^{2}_{ \overline{Y} }( Y )
    <
    \infty.
\end{equation}
The number $\mathrm{Area}( \phi )$ is called the \emph{parametrized area} of $\phi$.

Lytchak and Wenger consider in \cite[Section 1.2]{Lyt:Wen:17} the class $\Lambda( \partial Y, \overline{Y} )$ of those $u \in N^{1,2}( \mathbb{D}; \overline{Y} )$ whose trace $u' \colon \mathbb{S}^{1} \rightarrow \overline{Y}$ is a (weakly) monotone parametrization of $\partial Y$. Associated to such maps, one defines $\mathrm{Area}( u )$ by integrating a Jacobian of $u$ \cite[Section 1.2]{Lyt:Wen:17}. Also, $E( u ) = \int_{ \mathbb{D} } \rho_{u}^{2} \,d\mathcal{L}^{2}$ is the corresponding energy.

\Cref{thm:carat:metric} implies that every quasiconformal homeomorphism $\phi \colon \mathbb{D} \rightarrow Y$ defines an element of $\Lambda( \partial Y, \overline{Y} )$. If $\phi$ is $K$-quasiconformal, then $E( \phi ) \leq K \mathrm{Area}( \phi ) < \infty$ due to \eqref{eq:boundedarea}. Then \cite[Theorem 7.6]{Lyt:Wen:17} yields the existence of $u_{e} \in \Lambda( \partial Y, \overline{Y} )$ with minimal $E(u)$ among all $u \in \Lambda( \partial Y, \overline{Y} )$, referred to as an energy minimizer. Similarly, \cite[Theorem 1.1]{Lyt:Wen:17} yields the existence of $u_{a} \in \Lambda( \partial Y, \overline{Y} )$ of minimal parametrized area.

For a general quasiconformal Jordan domain $Y$, it is not clear whether or not $u_{e}$ (or $u_{a}$) is a quasiconformal homeomorphism. However, if we also assume that $\overline{Y}$ is geodesic, $\partial Y$ is rectifiable, and $\overline{Y}$ satisfies a quadratic isoperimetric inequality and \eqref{eq:pointshavezeromod} at the boundary points of $Y$, the energy minimizer $u_e$ is a quasiconformal homeomorphism \cite[Theorem 1.3]{Cre:Rom:20}. We refer the interested reader to \cite{Lyt:Wen:17} and \cite{Cre:Rom:20} for further reading.

There are some ways to construct quasiconformal Jordan domains. For example, if $X$ is a metric surface satisfying local versions of annular linear local connectivity and Ahlfors $2$-regularity, Theorem 4.17 of \cite{Wil:10} yields the existence of many quasiconformal Jordan domains $Y \subset X$ satisfying the assumptions of \Cref{prop:planar:QC}. Some examples of quasiconformal Jordan domains can also be obtained from \cite[Theorem 2]{Hai:09}.

\subsection*{Acknowledgement}The author was supported by the Academy of Finland, project number 308659 and by the Vilho, Yrjö and Kalle Väisälä Foundation. The author wishes to thank the referee for their helpful suggestions.

\bibliographystyle{alpha}
\bibliography{bibliography} 

\end{document}